\documentclass[11pt]{amsart}
\usepackage[top=1in, bottom=1in, left=1in, right=1in]{geometry}

\usepackage{mathtools,mathrsfs,amsthm,amsfonts,amssymb,graphicx,appendix,cancel}
\usepackage{xspace,enumerate,comment,bbm,nth}

\usepackage{tikz,pgfplots}
\usepackage{hyperref}
\usepackage[normalem]{ulem}   
\usepackage[active]{srcltx}

\DeclareMathOperator*{\esssup}{ess\,sup}

\numberwithin{equation}{section}
\newtheorem{theorem}{\sc Theorem}[section]

\newtheorem{definition}[theorem]{\sc Definition}
\newtheorem{lemma}[theorem]{\sc Lemma}

\theoremstyle{plain}

\newtheorem{prop}[theorem]{\sc Proposition}

\theoremstyle{remark}

 \newcommand{\ol}[1]{\overline{#1}}
 \newcommand{\ul}[1]{\underline{#1}}

\newcommand{\F}{\mathcal F}
\newcommand{\M}{\mathcal M}
\newcommand{\R}{\mathbb{R}}
\newcommand{\Sol}{\mathcal S}

\newcommand{\superSol}{\overline{\mathfrak{S}}}
\newcommand{\Z}{\mathbb{Z}}
\newcommand{\N}{\mathbb{N}}

\renewcommand{\P}{\mathbb{P}}
\newcommand{\EE}{\mathbb{E}}

\newcommand{\cal}[1]{\mathcal #1}
\newcommand{\1}{\mathbbm1}
\newcommand{\HV}{{\mathcal H}}
\newcommand{\HF}{{\mathcal H}}

\newcommand{\eps}{\varepsilon}
\renewcommand{\epsilon}{\varepsilon}

\renewcommand{\emptyset}{\varnothing}

\renewcommand{\ge}{\geqslant}
\newcommand{\CC}{\D C}
\newcommand{\Lip}{\D{Lip}}
\newcommand{\landa}{\lambda}

\newcommand{\ccyl}{[0,+\infty)\times\R}

\newcommand{\Ham}{\mathscr{H}(\alpha_0,\alpha_1,\gamma)}
\newcommand{\Hamall}{\mathscr{H}}

\newcommand{\Hamqc}{\mathscr{H}_{qc}(\alpha_0,\alpha_1,\gamma)}
\newcommand{\Hamsqc}{\mathscr{H}_{sqc}(\alpha_0,\alpha_1,\gamma,\eta)}

\newcommand{\D}[1]{\mbox{\rm #1}}

\begin{document}
	
\title[Stochastic homogenization of quasiconvex degenerate viscous HJ
equations in 1d]{Stochastic homogenization of quasiconvex\\ degenerate viscous HJ equations in 1d}

\author[A.\ Davini]{Andrea Davini}
\address{Andrea Davini\\ Dipartimento di Matematica\\ {Sapienza} Universit\`a di
  Roma\\ P.le Aldo Moro 2, 00185 Roma\\ Italy}
\email{davini@mat.uniroma1.it}
%\urladdr{https://www1.mat.uniroma1.it/people/davini/home.html}

\date{July 12, 2023}

 % 35B27 Homogenization; 35F21 Hamilton-Jacobi equations; 60G10 Stationary stochastic processes; 60K37 Processes in random environments.

\begin{abstract}
We prove homogenization for degenerate viscous Hamilton-Jacobi equations in dimension one in stationary ergodic environments
with a quasiconvex and superlinear Hamiltonian of fairly general type. We furthermore show that the effective Hamiltonian is quasiconvex. This latter result is new even in the periodic setting, despite homogenization has been known for quite some time.   
\end{abstract}

\maketitle

\section{Introduction}\label{sec:intro}

We are concerned with the asymptotic behavior,  as $\epsilon\to 0^+$, of solutions of a  viscous
Hamilton-Jacobi (HJ) equation of the form
\begin{equation}\label{eq:introHJ}
  \partial_t u^\epsilon=\epsilon a\left({x}/{\epsilon},\omega\right) \partial^2_{xx} u ^\epsilon+H\left({x}/{\eps},\partial_x u^\epsilon,\omega\right) \qquad \hbox{in $(0,+\infty)\times\R$,}
\end{equation}
when $H:\R\times\R\times\Omega\to\R$ is assumed superlinear and quasiconvex in the momentum, and $a:\R\times\Omega\to [0,1]$ is Lipschitz on $\R$ for every fixed $\omega$. The dependence of the equation on the random environment $(\Omega,\F,\P)$ enters through the Hamiltonian $H(x,p,\omega)$ and the diffusion coefficients $a(x,\omega)$, which are assumed to be stationary with respect to shifts in $x$, and bounded and Lipschitz continuous on $\R$, for every fixed $p\in\R$ and $\omega\in\Omega$.  The diffusion coefficient is assumed not identically zero,  but we require $\min_\R a(\cdot,\omega)=0$, at least almost surely.  Under these assumptions, we prove homogenization for equation \eqref{eq:introHJ} and we show that the effective Hamiltonian is also quasiconvex. 
This latter result is new in the literature even in the periodic setting, despite homogenization has been known for quite some time \cite{ Evans92,  LS_almostperiodic}. 
%This latter result is new in the literature even in the periodic case, despite homogenization of \eqref{eq:introHJ} in this setting has been  known for quite some time \cite{ Evans92,  LS_almostperiodic}. 
%
We refer to Section \ref{sec:result} for a complete description of the set of conditions assumed in the paper and for the precise statement of our homogenization results.\smallskip

Equations of the form \eqref{eq:introHJ} are a subclass of general
viscous stochastic HJ equations
\begin{equation}\label{vHJ}
  \partial_tu^\epsilon= \epsilon\text{tr}\left(A\left(\frac{x}{\epsilon},\omega\right) D^2 u^\epsilon\right)+H\left(D u^\epsilon,\frac{x}{\epsilon},\omega\right)\qquad \hbox{in $(0,+\infty)\times\R^d$,}
\end{equation}
where $A(\cdot,\omega)$ is a bounded, symmetric, and nonnegative
definite $d\times d$ matrix with a Lipschitz square root. The ingredients of the equation are assumed to be stationary with respect to shifts in $x$. 
This setting encompasses the periodic and the quasiperiodic cases, for which homogenization has been proved under fairly general assumptions by showing the existence of (exact or approximate) correctors, i.e., sublinear functions that solve an associated stationary HJ equations \cite{LPV, Evans92, I_almostperiodic, LS_almostperiodic}. 
In the stationary ergodic setting, such solutions do not exist in general, as it was shown in \cite{LS_correctors}, see also \cite{DS09,CaSo17} for inherent discussions and results. This is the main reason why the extension of the homogenization theory to random media is nontrivial 
and required the development of new arguments. 

The first homogenization results in this framework were obtained for convex Hamiltonians in the case of inviscid equations in \cite{Sou99, RT} and then for their viscous counterparts in \cite{LS2005,KRV}. 
By exploiting the metric character of first order HJ equations, homogenization has been extended to the case of quasiconvex Hamiltonians, first in dimension 1 \cite{DS09} and then in any space dimension \cite{AS}. 

The topic of homogenization in stationary ergodic media for HJ equations that are nonconvex in the gradient variable
remained an open problem for about fifteen years. Recently, it has been shown via counterexamples, first in the inviscid case \cite{Zil,FS}, then in the viscous one \cite{FFZ}, that homogenization can fail for Hamiltonians of the form $H(x,p,\omega):=G(p)+V(x,\omega)$ whenever $G$ has a strict saddle point. This has shut the door to the possibility of having a general qualitative homogenization theory  in the stationary ergodic setting in dimension $d\geqslant 2$, at least without imposing further mixing conditions on the stochastic environment. 

On the positive side, a quite general homogenization result, which includes as particular instances both the inviscid and viscous cases, has been established in \cite{AT} for Hamiltonians that are positively homogeneous of degree $\alpha\geqslant 1$ and under a finite range of dependence condition on the probability space $(\Omega,\F, \P)$. 

In the inviscid case, homogenization has been proved in \cite{ATY_1d,Gao16} in dimension $d=1$ for a rather general class of coercive and nonconvex Hamiltonians, and in any space dimension for Hamiltonians of the form $H(x,p,\omega)=\big(|p|^2-1\big)^2+V(x,\omega)$, see   \cite{ATY_nonconvex}.\smallskip 

Even though the addition of a diffusive term is not expected to prevent homogenization, the literature on viscous HJ equations 
in stationary ergodic media has remained rather limited until very recently. 
The viscous case is in fact known to present additional 
challenges which cannot be overcome by mere
modifications of the methods used for $a\equiv 0$.

Apart from already mentioned work \cite{AC18}, 
several progresses concerning the homogenization of the viscous HJ equation \eqref{vHJ} with
nonconvex Hamiltonians have been recently made in 
\cite{DK17, YZ19, KYZ20, Y21b, DKY23, D23a}. In the joint paper \cite{DK17}, we have
shown homogenization of \eqref{vHJ} with $H(x,p,\omega)$ which
are ``pinned'' at one or several points on the $p$-axis and convex in
each interval in between. For example, for every $\alpha>1$ the
Hamiltonian $H(x,p,\omega)=|p|^\alpha-c(x,\omega)|p|$ 
is pinned at $p=0$ (i.e.\
$H(0,x,\omega)\equiv
\mathrm{const}$) and convex in $p$ on each of the two intervals $(-\infty, 0)$ and
$(0,+\infty)$.
Clearly, adding a non-constant potential $V(x,\omega)$ breaks the pinning
property. In particular, homogenization of equation \eqref{vHJ} for $d=1$, $A\equiv \mathrm{const}>0$ and 
%\begin{equation}
 % \label{open}
 $H(x,p,\omega):=\frac12\,|p|^2-c(x,\omega)|p|+V(x,\omega)$ with $c(x,\omega)$ bounded and strictly positive 
 % 0<c(x,\omega)\le C,\quad \ \beta>0
%\end{equation}
remained an open problem even when $c(x,\omega)\equiv c>0$.  
Homogenization for this kind of equations with $A\equiv 1/2$ and $H$ as
above with $c(x,\omega)\equiv c>0$ was proved in \cite{KYZ20} 
under a novel hill and valley condition,\footnote{Such a hill and valley condition is fulfilled for a wide class of typical random environments without any restriction on their mixing properties, see  \cite[Example 1.3]{KYZ20} and \cite[Example 1.3]{YZ19}.  It is however not satisfied if the potential is ``rigid'', for example, in the periodic case.}
 that was introduced in \cite{YZ19} 
to study a sort of discrete version of this problem, where the Brownian motion in the  
stochastic control problem associated with equation \eqref{eq:introHJ} is replaced by controlled random walks on $\Z$. 
The approach of \cite{YZ19,KYZ20} relies on the Hopf-Cole
transformation, stochastic control representations of solutions and
the Feynman-Kac formula.  
It is applicable to \eqref{eq:introHJ}
with $H(x,p,\omega):=G(p)+V(x,\omega)$ and $G(p):=\frac12|p|^2-c|p|=\min\{\frac12p^2-cp,\frac12p^2+cp\}$ only. 
In the joint paper 
\cite{DK22}, we have proposed a different proof solely based on PDE methods. 
This new approach is flexible enough to be applied to the possible degenerate case 
$a(x,\omega)\ge 0$, and to any $G$ which is a minimum of a finite number of convex
superlinear functions $G_i$  having the same
minimum. The original hill and valley condition assumed in \cite{YZ19,KYZ20} is weakened in 
favor of a scaled hill and valley condition.\footnote{We stress that such scaled hill (respectively, valley) condition, which was assumed in \cite{DK22} and subsequently in \cite{D23a}, is empty when $a\equiv0$, namely it does not impose any restriction on the potential, or, equivalently, on the random environment. In particular, the inviscid periodic setting is covered by these works.} The arguments crucially rely on this condition and on the fact that all the functions $G_i$ have the same minimum.

%As in \cite{KYZ20}, the shape of the effective Hamiltonian associated with 
%$G$ is derived from the ones associated to each $G_i$.

The PDE approach introduced in \cite{DK22} was subsequently refined in \cite{Y21b} in order to prove 
homogenization for equation \eqref{eq:introHJ} when the function $G$ is superlinear and quasiconvex,  $a>0$ and the pair 
$(a,V)$ satisfies a scaled hill condition equivalent to the one adopted in \cite{DK22}. 
The core of the proof consists in showing existence and uniqueness of correctors  that possess stationary derivatives satisfying suitable bounds. These kind of results were obtained in \cite{DK22} by proving tailored-made  comparison principles and by exploiting a general result from \cite{CaSo17}  (and this was the only point where the piecewise convexity of $G$ was used). The novelty brought in by \cite{Y21b} 
relies on the nice observation that this can be directly proved via ODE arguments, which are viable since we are in one space dimension and $a>0$. 

By reinterpreting this approach in the viscosity sense and by making a more substantial use of viscosity techniques, the author has extended in \cite{D23a} to the possible degenerate case $a\geqslant 0$ the homogenization result established in \cite{Y21b}, providing in particular a unified proof which encompasses both the inviscid and the viscous case.  
%A key tool in this regard is played by the well known stability property of the notion of viscosity solution. 
%Our exposition also benefits of some observations and arguments developed by the author in the attempt to prove a more general homogenization result.  They have already displayed their usefulness for the joint research \cite{DKY23}, where they appeared for the first time. Their use here has permitted to simplify some proofs. %\vspace{0.5ex}

A substantial improvement of the results of \cite{DK22}  is provided in our joint work 
 \cite{DKY23}, where the  
main novelty consists in allowing a  general
superlinear $G$ without any restriction on its shape or the
number of its local extrema. Our analysis crucially relies on the assumption that the pair 
$(a,V)$ satisfies the scaled hill and valley condition and, differently from \cite{DK22}, 
that the diffusive coefficient $a$ is nondegenerate, i.e., $a>0$. 
%
%Analogously to \cite{KYZ20, DK22}, the
%function $G$ can be seen as a minimum of quasiconvex superlinear
%functions $G_i$, but when these latter have distinct minima there is
%a nontrivial interaction among them in the homogenization process, and
%the shape of the effective Hamiltonian associated with $G$ can no
%longer be guessed from those associated with each $G_i$. 
The proof consists in showing existence of suitable correctors whose derivatives 
are confined on different branches of $G$, as well as on a strong induction argument 
which uses as  base case the homogenization result established in \cite{Y21b}.\vspace{0.5ex}
%The fact that homogenization for quasiconvex $G$ was only available for $a>0$ is one of the  
%main reasons why we did not try to generalize our arguments to the possible 
%degenerate case $a\geqslant 0$.\vspace{0.5ex}

The novelty of the current work consists in dealing with quasiconvex Hamiltonians $H(x,p,\omega)$ where the dependence on $x$ and $p$ is not necessarily decoupled and 
under the only assumption that the diffusive coefficient $a$ is not identically zero and vanishes at some points or on some regions of $\R$, at least almost surely. 
This represents a significant step forward in the direction of obtaining a general homogenization result for viscous HJ equations in dimension 1 in the stationary ergodic setting. As a starting step for our analysis, we remark that we can restrict to consider stationary Hamiltonians that are strictly quasiconvex and superquadratic in the momentum variable. This class of Hamiltonians is in fact dense, in an appropriate sense, in the wider class of quasiconvex Hamiltonians we are considering, see Section \ref{sec:outline} and Appendix \ref{app:density} for more details, and, as it is known, homogenization is preserved in the limit. A first novelty with respect to previous works on the subject  
\cite{KYZ20,DK22,Y21b,DKY23,D23a} is that the expected minimum of the effective Hamiltonian cannot be deduced in a direct way from the bounds enjoyed by $H$. As we will show in Section \ref{sec:homogenization}, such a minimum coincides with a critical value $\lambda_0$ which is associated to the pair $(a, H)$ in a more intrinsic way, see Section \ref{sec:correctors}.  For the definition of $\landa_0$ and the analysis of its relevant properties, we use in a crucial way the assumption that $H$ is superquadratic in the momentum in order to exploit known H\"older regularity results for continuous supersolutions of the ``cell'' equation associated with \eqref{eq:introHJ}. This part is independent on the quasiconvexity condition of the Hamiltonian. 

Our approach to prove homogenization for equation  \eqref{eq:introHJ} is inspired by the one introduced in \cite{Y21b}, where the main idea was that of proving existence and uniqueness of suitable correctors of an associated cell problem. The fact that the diffusion coefficient $a(\cdot,\omega)$ vanishes on $\R$, almost surely, and that $H$ is strictly quasiconvex allows us to established suitable existence, uniqueness and regularity results for correctors in our setting. These kind of results rely on a fine qualitative analysis on the behavior of the derivatives of deterministic solutions to the cell equation when we approach the zero set  of the diffusion coefficient. The main difficulty relies on the fact that we do not make any  assumption on the structure of this set, besides those implied by the stationary ergodicity of $a(x,\omega)$.

A separate analysis is required to show homogenization at the bottom and the existence of a possible flat part there. This is achieved by proving suitable matching upper and lower bounds. This point is handled here via two novel ideas, corresponding to Theorems \ref{teo1 upper qc} and \ref{teo lower bound}, where we provide two different arguments to suitably bridge a pair of different viscosity solutions of the same cell equation. In Theorem \ref{teo1 upper qc} this is pursued by crucially exploiting the quasiconvexity  of $H$ and the fact that $a(\cdot,\omega)$ vanishes at some point, at least almost surely, in order to derive a quasiconvexity-type property of the expected effective Hamiltonian.  Theorem  \ref{teo lower bound} instead does not use directly the quasiconvexity of $H$. The core of the proof consists in 
exploiting nonexistence of supersolutions to the cell equation at levels strictly below the critical constant $\landa_0$ to construct a lift which allows to gently transition between two solutions of the same cell equation. \vspace{0.5ex}

The paper is organized as follows. In Section \ref{sec:result} we present the setting and the standing assumptions and we state our homogenization results, see Theorems \ref{thm:genhom} and \ref{thm:genhom2}. In Section \ref{sec:outline} we explain how Theorem \ref{thm:genhom} can be derived from Theorem  \ref{thm:genhom2}, then we 
outline the strategy we will follow to prove  Theorem  \ref{thm:genhom2}.  In Section \ref{sec:correctors} we give the definition of the critical value $\landa_0$, which is characterized as the minimal value of $\landa$ for which the correspondig corrector equation can be solved, and we correspondingly show existence of deterministic solutions of such equations. Section \ref{sec:deterministic  correctors} contains a fine analysis concerning uniqueness and regularity properties of suitable deterministic solutions of the corrector equations we are considering. Section \ref{sec:homogenization}  contains the proof of Theorem \ref{thm:genhom2}. The analysis for the flat part at the bottom mentioned above is contained in Section \ref{sec:flat part}.
The paper ends with two appendices. In Appendix \ref{app:density}, we prove the density property stated in Proposition \ref{prop density} which is at the base of the reduction argument described in Section \ref{sec:outline}. In Appendix \ref{app:PDE} we have collected all the PDE results we need for our analysis.\medskip

\indent{\textsc{Acknowledgements. $-$}} The author wishes to thank Luca Rossi for the help provided with the proof of Proposition \ref{prop pointwise solution}. 

\section{Preliminaries}\label{sec:preliminaries}
\subsection{Assumptions and statement of the main results}\label{sec:result}

We will denote by $\CC(\R)$ and $\CC(\R\times\R)$ the Polish spaces of continuous functions  on $\R$ and on $\R\times\R$, endowed with a metric inducing the topology of uniform convergence on compact subsets of $\R$ and of $\R\times\R$, respectively. 

The triple $(\Omega,\F, \P)$ denotes a probability space, where $\Omega$ is a Polish space, ${\cal F}$ is the $\sigma$-algebra of Borel subsets of $\Omega$, and $\P$ is a complete probability measure on $(\Omega,{\cal F})$.\footnote{The assumptions that $\Omega$ is a Polish space and $\P$ is a complete probability measure are only used in the proofs of Lemma \ref{lemma lambda zero} and Theorem \ref{teo1 random correctors} to show measurability of the random objects therein considered. The assumption that $\Omega$ is a Polish space is also used in the proof of Lemma \ref{appA lemma measurable selection} in order to apply  the  Kuratowski–Ryll-Nardzewski measurable selection theorem \cite{KuRy65}.} We will denote by ${\cal B}$ the Borel $\sigma$-algebra on $\R$ and equip the product space $\R\times \Omega$ with the product $\sigma$-algebra ${\mathcal B}\otimes {\cal F}$.

We will assume that $\P$ is invariant under the action of a one-parameter group $(\tau_x)_{x\in\R}$ of transformations $\tau_x:\Omega\to\Omega$. More precisely, we assume that the mapping
$(x,\omega)\mapsto \tau_x\omega$ from $\R\times \Omega$ to $\Omega$ is measurable, $\tau_0=id$, $\tau_{x+y}=\tau_x\circ\tau_y$ for every $x,y\in\R$, and $\P\big(\tau_x (E)\big)=\P(E)$ for every $E\in{\cal F}$ and $x\in\R$. We will furthermore assume that the action of $(\tau_x)_{x\in\R}$ is {\em ergodic}, i.e., any measurable function $\varphi:\Omega\to\R$ satisfying $\P(\varphi(\tau_x\omega) = \varphi(\omega)) = 1$ for every fixed $x\in{\R}$ is almost surely equal to a constant.
 If $\varphi\in L^1(\Omega)$, we write $\EE(\varphi)$ for the mean of $\varphi$ on $\Omega$, i.e. the quantity
$\int_\Omega \varphi(\omega)\, d \P(\omega)$.\smallskip

In this paper, we will consider an equation of the form 
\begin{equation}\label{eq:generalHJ}
\partial_t u=a(x,\omega) \partial^2_{xx} u +H(x,\partial_x u,\omega),\quad\hbox{in}\ (0,+\infty)\times\R,
\end{equation}
where the diffusion coefficient $a:\R\times\Omega\to [0,1]$ and the Hamiltonian $H:\R\times\R\times\Omega\to\R$ are assumed to be jointly measurable in all variables and {\em stationary}, meaning that 
\[
a(x+y,\omega)=a(x,\tau_y\omega)\quad\hbox{and}\quad H(x+y,p,\omega)=H(x,p,\tau_y \omega)
\qquad
\hbox{for all  $x,y,p\in\R$ and $\omega\in\Omega$.}
\]
The diffusion coefficient $a$ is furthermore assumed to satisfy the following conditions,  
for some constant $\kappa > 0$:
\begin{itemize}
\item[(A1)] \quad $a(\cdot,\omega)\not\equiv 0$\ \ and\ \  $\min_\R a(\cdot,\omega)=0$ \ on $\R$\quad for $\P$-a.e. $\omega\in\Omega$;\smallskip
\item[(A2)]  \quad $\sqrt{a(\,\cdot\,,\omega)}:\R\to [0,1]$\ is $\kappa$--Lipschitz continuous for all $\omega\in\Omega$.\footnote{
Note that (A2) implies that $a(\cdot,\omega)$ is $2\kappa$--Lipschitz in $\R$ for all $\omega\in\Omega$.  Indeed, 
for all $x,y\in\R$ we have 
\[
|a(x,\omega) - a(y,\omega)| = |\sqrt{a(x,\omega)} + \sqrt{a(y,\omega)}||\sqrt{a(x,\omega)} - \sqrt{a(y,\omega)}| \leqslant 2\kappa|x-y|.
\]
}
\end{itemize}
As for the Hamiltonian $H:\R\times\R\times\Omega\to\R$, we will assume $H(\cdot,\cdot,\omega)$ to belong to the family  $\Ham$ defined as follows, for some constants $\alpha_0,\alpha_1>0$ and $\gamma>1$ independent of $\omega\in\Omega$. 
\begin{definition}\label{def:Ham}
	A function $H:\R\times\R\times\Omega\to\R$ is said to be in the class $\Ham$ if it satisfies the following conditions, for some constants $\alpha_0,\alpha_1>0$ and $\gamma>1$:
	\begin{itemize}
		\item[(H1)] \quad $\alpha_0|p|^\gamma-1/\alpha_0\leqslant H(x,p,\omega)\leqslant\alpha_1(|p|^\gamma+1)$\quad for all ${(x,p,\omega)\in\R\times\R\times\Omega}$;\medskip
		\item[(H2)]\quad $|H(x,p,\omega)-H(x,q,\omega)|\leqslant\alpha_1\left(|p|+|q|+1\right)^{\gamma-1}|p-q|$\quad for all $(x,\omega)\in\R\times\Omega$ and $p,q\in\R$;\medskip
	\item[(H3)] \quad $|H(x,p,\omega)-H(y,p,\omega)|\leqslant\alpha_1(|p|^\gamma+1)|x-y|$\quad 
	for all $x,y,p\in\R$ and $(p,\omega)\in\R\times\Omega$.\medskip		
\end{itemize}
%We will denote by $\Hamall$ the union of the families $\Ham$, where $\alpha_0,\alpha_1$ vary in $(0,+\infty)$ and $\gamma$ in $(1,+\infty)$. 
\end{definition}

Solutions, subsolutions and supersolutions of \eqref{eq:generalHJ}
will be always understood in the viscosity sense, see \cite{CaCa95, users,barles_book,bardi}, and implicitly assumed continuous, without any further specification. Assumptions (A2) and (H1)-(H3)
guarantee well-posedness in $\D{UC}(\ccyl)$ of the Cauchy problem for
the parabolic equation \eqref{eq:generalHJ}, as well as Lipschitz
estimates for the solutions under appropriate assumptions on the
initial condition, see Appendix \ref{app:PDE} for more
details.

      The purpose of this paper is to prove a homogenization
      result for equation \eqref{eq:introHJ} when the nonlinearity $H$ is additionally assumed quasiconvex, meaning that 
\begin{itemize}
\item[(qC)] 
$H(x,(1-s) p_0+sp_1,\omega)
\leqslant\max\{H(x,p_0,\omega),H(x,p_1,\omega)\}
\quad
\hbox{for all $p_0\not=p_1\in\R$ and $s\in (0,1)$,}$
\end{itemize}
for every $(x,\omega)\in\R\times\Omega$. 
We remark that the quasiconvexity assumption has the following equivalent geometric meaning: 
\[
\{p\in\R\,:\,H(x,p,\omega)\leqslant \lambda\}\quad\hbox{is a (possibly empty) convex set for all $\lambda\in\R$ and $(x,\omega)\in\R\times\Omega$.} 
\]
We will denote by $\Hamqc$ the family of Hamiltonians $H\in\Ham$ satisfying condition (qC).
%
%Next, we introduce the general class of continuous and coercive functions $G$ that we consider in this paper.
%
%\begin{definition}\label{def:Ham}
%	A function $G:\R\to\R$ is said to be in the class $\Ham$ if it satisfies the following conditions for some constants $\alpha_0,\alpha_1>0$ and $\gamma>1$:
%	
%	\begin{itemize}
%		\item[(G1)] $\alpha_0|p|^\gamma-1/\alpha_0\leqslant G(p)\leqslant\alpha_1(|p|^\gamma+1)$\ for all ${p\in\R}$;\medskip
%		\item[(G2)] $|G(p)-G(q)|\leqslant\alpha_1\left(|p|+|q|+1\right)^{\gamma-1}|p-q|$\ for all $p,q\in\R$.
%	\end{itemize}
%      \end{definition}
%
Our main result reads as follows.

\begin{theorem}\label{thm:genhom}
Suppose $a$ satisfies (A1)-(A2) and $H(\cdot,\cdot,\omega)\in\Hamqc$  for some constants $\alpha_0,\alpha_1>0$ and $\gamma>1$, for every $\omega\in\Omega$. 
Then, %, for $\P$-a.e.\ $\omega$, as $\epsilon\to0$,  when subject to uniformly continuous initial data,
the viscous HJ equation \eqref{eq:introHJ} homogenizes, i.e., there exists a continuous and coercive function 
$\HV(H):\R\to\R$, called {\em effective Hamiltonian}, and a set $\hat\Omega$ of probability 1 such that, for every uniformly continuous function $g$ on $\R$ and every $\omega\in\hat\Omega$, the solutions $u^\epsilon(\cdot,\cdot,\omega)$ of \eqref{eq:introHJ} satisfying $u^\epsilon(0,\,\cdot\,,\omega) = g$ converge,
    locally uniformly on $[0, +\infty)\times \R$ as $\epsilon\to 0^+$, to the unique solution $\ol{u}$ of 
    \begin{eqnarray*}%\label{effeq}    
   \begin{cases}
    \partial_t \ol{u} = \HV(H)(\partial_x\ol{u}) & \hbox{in $(0,+\infty)\times\R$}\\
    \ol{u}(0,\,\cdot\,) = g & \hbox{in $\R$}.
    \end{cases}
    \end{eqnarray*}
Furthermore, $\HV(H)$ satisfies (H1), is locally Lipschitz and quasiconvex. 
\end{theorem}

The effective Hamiltonian $\HV(H)$ also depends on the diffusion coefficient $a$, but since the latter  
will remain fixed throughout the paper, we will not keep track of this in our notation.\smallskip 

In order to prove Theorem \ref{thm:genhom}, we find convenient to restrict to Hamiltonians $H\in\Ham$ with $\gamma>2$ (and $\alpha_0,\,\alpha_1>0$) and which are  {\em strictly quasiconvex}, meaning that they satisfy (qC) above with a strict inequality. Such a strictly quasiconvex Hamiltonian $H$ has a unique minimizer $\hat p(x,\omega)$ of $H(x,\cdot,\omega)$ in $\R$, for every fixed $(x,\omega)\in\R\times\Omega$. For every fixed $\eta>0$, we introduce the family $\Hamsqc$ of  strictly quasiconvex Hamiltonians $H\in\Ham$ satisfying the following condition, dependent on $\eta$:
\begin{itemize}
\item[(sqC)] \quad for all $(x,\omega)\in\R\times\Omega$, 
\begin{eqnarray*}
H(x,p_1,\omega)-H(x,p_2,\omega)\geqslant \eta|p_1-p_2|\qquad &&\hbox{for all $p_1<p_2\leqslant \hat p(x,\omega)$,}\\
H(x,p_2,\omega)-H(x,p_1,\omega)\geqslant \eta|p_1-p_2| \qquad && \hbox{for all $p_2>p_1\geqslant \hat p(x,\omega)$.}
\end{eqnarray*}
\end{itemize}
We will prove the following result. 

\begin{theorem}\label{thm:genhom2}
Suppose $a$ satisfies (A1)-(A2) and $H(\cdot,\cdot,\omega)\in\Hamsqc$  for some constants $\alpha_0,\alpha_1>0$, $\gamma>2$ and $\eta>0$, for every $\omega\in\Omega$. 
Then the viscous HJ equation \eqref{eq:introHJ} homogenizes, in the sense of Theorem \ref{thm:genhom}, and the effective Hamiltonian 
$\HV(H)$ satisfies (H1), is locally Lipschitz and quasiconvex. 
Furthermore, there exist $\lambda_0\in\R$ and $\theta^-(\lambda_0)\leqslant \theta^+(\lambda_0)$ in $\R$ such that 
\[
\HV(H)(\theta)=\lambda_0=\min_\R \HV(H)\qquad\hbox{for all $\theta\in [\theta^-(\landa_0),\theta^+(\landa_0)]$},
\] 
$\HV(H)$ is strictly decreasing in  $(-\infty,\theta^-(\landa_0))$  and is strictly increasing in $(\theta^+(\landa_0),+\infty)$.
\end{theorem}
 
\subsection{Outline of the proof strategy}\label{sec:outline}

In this section, we outline the strategy that we will follow to prove the homogenization results stated in Theorems \ref{thm:genhom} and  \ref{thm:genhom2}. 
In particular, we will explain how we can reduce to prove Theorem \ref{thm:genhom2} only.\smallskip

Let us denote by $u_\theta(\,\cdot\,,\,\cdot\,,\omega)$ the unique Lipschitz solution to \eqref{eq:generalHJ} with initial condition $u_\theta(0,x,\omega)=\theta x$ on $\R$, and let us introduce the following deterministic quantities, defined almost surely on $\Omega$, see \cite[Proposition 3.1]{DK22}:
\begin{eqnarray}\label{eq:infsup}
	\HV^L(H) (\theta):=\liminf_{t\to +\infty}\ \frac{u_\theta(t,0,\omega)}{t}\quad\text{and}\quad 
	\HV^U(H) (\theta):=\limsup_{t\to +\infty}\ \frac{u_\theta(t,0,\omega)}{t}.
\end{eqnarray}
In view of \cite[Lemma 4.1]{DK17}, proving homogenization amounts to showing that $\HV^L(H)(\theta) = \HV^U(H)(\theta)$ for every $\theta\in\R$. If this occurs, their common value is denoted by $\HV(H)(\theta)$. The function $\HV(H):\R\to\R$ is called the effective Hamiltonian associated with $H$. It has already appeared in the statement of Theorems \ref{thm:genhom} and \ref{thm:genhom2}.\smallskip

The reduction argument mentioned above can be precisely stated as follows. 

\begin{prop}\label{prop reduction}
Suppose Theorem \ref{thm:genhom2} holds for any $\alpha_0,\,\alpha_1>0$, $\gamma>2$ and $\eta>0$. Then Theorem \ref{thm:genhom} holds. 
\end{prop}

For this, we will need the following density result, whose proof is postponed to Appendix \ref{app:density}. 

\begin{prop}\label{prop density}
Let $H:\R\times\R\times\Omega\to\R$ be a stationary quasiconvex Hamiltonian such that $H(\cdot,\cdot,\omega)\in\Hamqc$ for some constants  $\alpha_0,\,\alpha_1>0$ and 
$\gamma>1$. There exist constants $\ol \alpha_0\geqslant \alpha_0$, $\ol\alpha_1\geqslant \alpha_1$, $\ol\gamma:=\max\{\gamma,4\}$, an infinitesimal sequence $(\eta_n)_n$ in $(0,1)$ and a sequence of stationary Hamiltonians $H_n:\R\times\R\times\Omega\to\R$ with $H_n(\cdot,\cdot,\omega)\in\Hamall_{sqc}(\ol\alpha_0,\ol\alpha_1,\ol\gamma,\eta_n)$ for every $\omega\in\Omega$ and $n\in\N$ such that 
\[
\lim_n\|H(\cdot,\cdot,\omega)-H_n(\cdot,\cdot,\omega)\|_{L^\infty\left(\R\times [-R,R]\right)}=0\qquad
\hbox{for every $R>0$ and $\omega\in\Omega$.}
\]
\end{prop}

\begin{proof}[Proof of Proposition \ref{prop reduction}]
Let $H$ be a stationary quasiconvex Hamiltonian satisfying $H(\cdot,\cdot,\omega)\in\Hamqc$  for some constants $\alpha_0,\alpha_1>0$ and $\gamma>1$, for every $\omega\in\Omega$. Let $(H_n)_n$ be the sequence of strictly quasiconvex Hamiltonians given by  Proposition \ref{prop density}.
%
%According to Proposition \ref{prop density}, there exist constants $\ol \alpha_0\geqslant \alpha_0$, $\ol\alpha_1\geqslant \alpha_1$, $\ol\gamma:=\max\{\gamma,4\}$, an infinitesimal sequence $(\eta_n)_n$ in $(0,1)$ and a sequence of stationary Hamiltonians $H_n:\R\times\R\times\Omega\to\R$ with $H_n(\cdot,\cdot,\omega)\in\Hamall_{sqc}(\ol\alpha_0,\ol\alpha_1,\ol\gamma,\eta_n)$ for every $\omega\in\Omega$ and $n\in\N$ such that 
%\[
%\lim_n\|H(\cdot,\cdot,\omega)-H_n(\cdot,\cdot,\omega)\|_{L^\infty\left(\R\times [-R-R]\right)}=0\qquad
%\hbox{for every $R>0$ and $\omega\in\Omega$.}
%\]
By Theorem \ref{thm:genhom2}, the viscous HJ equation \eqref{eq:introHJ} homogenizes for each $H_n$, with effective Hamiltonian $\HV(H_n)$ which is quasiconvex in $\R$, and strictly quasiconvex outside the possible flat part at the bottom. 
By Theorem \ref{appB teo stability},  the viscous HJ equation \eqref{eq:introHJ} homogenizes for $H$ too, with effective Hamiltonian $\HV(H)$ which is the local uniform limit in $\R$ of the effective Hamiltonians $\big(\HV(H_n)\big)_n$. From Proposition \ref{appB prop HF} we derive that $\HV(H)$ satisfies (H1) and is locally Lipschitz in $\R$, while the   quasiconvexity of $\HV(H)$ is inherited from that of the effective Hamiltonians $\HV(H_n)$. 
\end{proof}

Let us briefly outline the strategy to prove Theorem  \ref{thm:genhom2}. In order to prove that $\HV^L(\theta)=\HV^U(\theta)$ for all $\theta\in\R$, 
we will adopt the approach that was taken in \cite{KYZ20, DK22} and substantially  developed in \cite{Y21b}.
It consists in showing the existence of a viscosity Lipschitz solution $u(x,\omega)$ with stationary derivative
for the following stationary equation associated with \eqref{eq:generalHJ}, namely 
\begin{equation}\label{eq:cellPDE}
	a(x,\omega)u''+H(x,u',\omega)=\lambda\qquad\hbox{in $\R$}
		\tag{HJ$_\lambda(\omega)$}
\end{equation}
for every $\lambda\in\R$ and for $\P$-a.e. $\omega\in\Omega$. With a slight abuse of terminology, we will call such solutions {\em correctors}  in the sequel for the role they play in homogenization.\footnote{The word {\em corrector} is usually used in literature to refer to the function $u(x,\omega)-\theta\,x$\ with\ $\theta:=\EE[u'(0,\omega)]$, see for instance \cite{CaSo17} for a more detailed discussion on the topic.} 
In fact, the following holds.

\begin{prop}\label{prop consequence existence corrector}
Let $\lambda\in\R$ such that equation \eqref{eq:cellPDE} admits a viscosity solution $u$ with stationary gradient. 
Let us set $\theta:=\EE[u'(0,\omega)]$.  Then \ $\HV^L(\theta)=\HV^U(\theta)=\lambda$. 
\end{prop}

\begin{proof}
Let us set $F_\theta(x,\omega):=u(x,\omega)-\theta x$ for all $(x,\omega)\in\R\times\Omega$. Then $F_\theta(\cdot,\omega)$ 
is sublinear for $\P$-a.e. $\omega\in\Omega$, see for instance \cite[Theorem 3.9]{DS09}. Furthermore, it is a viscosity solution of \eqref{eq:cellPDE} with $H(\cdot,\theta+\cdot,\cdot)$ in place of $H$. The assertion follows arguing as in the proof of \cite[Lemma 5.6]{DK22}.
\end{proof}

The first step consists in identifying the set of $\landa\in\R$ for which equation \eqref{eq:cellPDE} can be solved, for every fixed $\omega\in\Omega$. In Section \ref{sec:correctors} we will show that this set is equal to the half-line $[\landa_0(\omega),+\infty)$, where $\landa_0(\omega)$ is a critical constant suitably defined. Furthermore, we will show the existence of a Lipschitz viscosity solution $u_{\lambda_0}(\cdot,\omega)$ to \eqref{eq:cellPDE} with $\landa:=\landa_0(\omega)$. 
For these kind of results, we will use in a crucial way the fact that $H(\cdot,\cdot,\omega)$ belongs to $\Ham$ with $\gamma>2$. This will allow us to use known H\"older regularity results for continuous supersolutions of equation \eqref{eq:cellPDE}, see Proposition \ref{prop Holder estimate} for more details. The stationarity of $a$ and $H$ and the ergodicity assumption imply that $\landa_0(\omega)$ is almost surely equal to a constant  $\lambda_0$.  This part is independent on the quasiconvexity condition of the Hamiltonian, which is therefore dropped. 

The subsequent step consists in showing existence of solutions $u^\pm_\landa(\cdot,\omega)$ to equation \eqref{eq:cellPDE} for $\landa>\landa_0(\omega)$ satisfying 
$(u^-_{\landa})'(\cdot,\omega)\leqslant (u_{\landa_0})'(\cdot,\omega) \leqslant (u^+_{\landa})'(\cdot,\omega)$ a.e. in $\R$. We will show that such solutions $u^\pm_{\landa}(\cdot,\omega)$ are unique, up to additive constants, and of class $C^1$, at least almost surely, in particular, by uniqueness, they possess stationary derivatives. These uniqueness and regularity results crucially depend on the fact that the Hamiltonian $H(\cdot,\cdot,\omega)$ is assumed strictly quasiconvex, in the sense of (sqC), and that $\min_\R a(\cdot,\omega)=0$, almost surely. They are established through a fine (deterministic) analysis which is performed in Section \ref{sec:deterministic  correctors}. It is also worth pointing out that the functions 
$u^\pm_\lambda$ are independent on the particular choice of the critical solution $u_{\landa_0}$.  

The last step consists in showing homogenization at the bottom and existence of a possible flat part there, namely 
$\HV^L(H)(\theta)=\HV^U(H)(\theta)=\lambda_0=\min_\R \HV(H)$ for all $\theta\in [\theta^-(\landa_0),\theta^+(\landa_0)]$, for suitably defined $\theta^-(\landa_0)\leqslant\theta^+(\landa_0)$. 
This point is handled here via two novel ideas, corresponding to Theorems \ref{teo1 upper qc} and \ref{teo lower bound}, where we provide two different arguments to suitably bridge a pair of distinct solutions of the same viscous HJ equation \eqref{eq:cellPDE}.  
In Theorem \ref{teo1 upper qc} this is pursued by crucially exploiting the quasiconvexity  of $H$ and the fact that 
$a(\cdot,\omega)$ vanishes at some points, at least almost surely, in order to derive a quasiconvexity-type property of $\HV^U(H)$. To prove Theorem  \ref{teo lower bound}, instead, we exploit the nonexistence of supersolutions of the viscous HJ equation \eqref{eq:cellPDE} at levels strictly below the critical constant $\landa_0(\omega)$ to construct a lift which allows to gently descend from $(u^+_{\landa})'(\cdot,\omega)$ to $(u^-_{\landa})'(\cdot,\omega)$, almost surely and for every fixed $\landa>\landa_0(\omega)$. The proof of Theorem  \ref{teo lower bound} does not use directly the quasiconvexity of $H$, but it exploits the fact that the functions $u^\pm_{\landa}(\cdot,\omega)$ are of class $C^1$ in $\R$, hence everywhere differentiable, for $\landa>\landa_0(\omega)$, at least almost surely. This latter property is proved in Section \ref{sec:deterministic correctors} by crucially exploiting the strict quasiconvexity condition (sqC).\smallskip 

We end this section with a disclaimer. As already precised in the previous section, sub and supersolutions of equation \eqref{eq:cellPDE} are understood in the viscosity sense and will be implicitly assumed continuous, without any further specification. We want to remark that, due to the positive sign of the diffusion term $a$, a viscosity supersolution (repectively, subsolution) $u$ to \eqref{eq:cellPDE} that is twice differentiable at $x_0$ will satisfy the  inequality
\[
a(x_0,\omega)u''(x_0,\omega)+H(x_0,u'(x_0),\omega)\leqslant \lambda.\qquad\hbox{(resp.,\  \ $\geqslant \lambda$.)}
\vspace{0.5ex}
\]

\section{The critical value}\label{sec:correctors}

In this section we will characterize the set of admissible $\landa\in\R$ for which the corrector equation 
\begin{equation}\label{eq cellPDE}
	a(x,\omega)u''+H(x,u',\omega)=\lambda\qquad\hbox{in $\R$}
		\tag{HJ$_\lambda(\omega)$}
\end{equation}
admits solutions. 
In this section we assume that $H$ is a stationary Hamiltonian belonging to the class  $\Ham$ for some fixed constants $\alpha_0,\alpha_1>0$ and $\gamma>2$, for every $\omega\in\Omega$. Instead, we will not assume quasiconvexity of $H$ here.\smallskip 

We will start by regarding at \eqref{eq cellPDE} as a deterministic equation, where $\omega$ is treated as a fixed parameter. The first step consists in characterizing the set of real $\lambda$ for which equation \eqref{eq cellPDE} admits (deterministic) viscosity solutions. To this aim, for every fixed $\omega\in\Omega$, we define a critical value associated with \eqref{eq cellPDE} defined as follows: 
\begin{equation}\label{def lambda zero}
\lambda_0(\omega):=\inf \left \{ \lambda\in\R\,:\,\hbox{equation \eqref{eq cellPDE} admits a continuous viscosity supersolution}\right\}.
\end{equation}
It can be seen as the lowest possible level under which supersolutions of equation \eqref{eq cellPDE} do not exist at all. 
Such a kind of definition is natural and appears in literature with different names, according to the specific problem for which it is introduced: ergodic constant, Ma\~ne critical value, generalized principal eigenvalue. In the inviscid stationary ergodic setting, it appears in this exact form in \cite{DS11}, 
%where it is termed {\em free critical value}, to make reference to the fact that in the definition we do not impose any restriction on the growth of the supersolutions we consider. We also quote  
see also \cite{IsSic20, CDD23} for an analogous definition in the noncompact deterministic setting.\smallskip 

%
%\begin{oss}
%This notion of critical value is somewhat natural for viscous and inviscid HJ equations. In the inviscid case (i.e., when $a\equiv 0$) and when $H$ is convex in $p$, it is also known as Ma\˜ne critical constant, see for instance \cite{FatSic} or  \cite{IsSic20, CDD23} for an analogous definition in the viscosity sense.
%\end{oss}
%

It is easily seen that $\lambda_0(\omega)$ is uniformly bounded from above on $\Omega$. 
Indeed, the function $u\equiv 0$ on $\R$ is a classical supersolution of \eqref{eq cellPDE} with 
$\lambda:=\sup_{(x,\omega)} H(x,0,\omega)\leqslant \alpha_1$. Hence we derive the upper bound
\begin{equation}\label{eq upper bound lambda zero}
\lambda_0(\omega)\leqslant \alpha_1\qquad\hbox{for all $\omega\in\Omega$.}
\end{equation}

We proceed to show that $\lambda_0(\omega)$ is also uniformly bounded from below and that equation \eqref{eq cellPDE} admits a solution for every $\omega\in\Omega$ and $\lambda\geqslant\lambda_0(\omega)$. For this, we need a preliminary lemma first.  

\begin{lemma}\label{lemma Busemann function}
Let $\omega\in\Omega$ and $\lambda>\lambda_0(\omega)$ be fixed. For every fixed $y\in\R$, there exists a viscosity supersolution $w_y$ of \eqref{eq cellPDE} satisfying $w_y(y)=0$ and 
\begin{equation}\label{claim visco elsewhere}
a(x,\omega)u''+H(x,u',\omega)=\lambda\qquad\hbox{in $\R\setminus\{y\}$}
\end{equation}
in the viscosity sense. In particular 
\begin{equation}\label{eq2 Holder bounds}
|w_y(x)-w_y(z)| \leqslant K |x-z|^{\frac{\gamma-2}{\gamma-1}}\qquad\hbox{for all $x,z\in\R$,}
\end{equation}
where $K=K(\alpha_0,\gamma,\lambda)>0$ is given explicitly by \eqref{eq Holder bound}. 
\end{lemma}

\begin{proof}
Let us denote by $\superSol(\lambda)(\omega)$ the family of continuous viscosity supersolutions of equation 
\eqref{eq cellPDE}. This set is nonempty since $\lambda>\lambda_0(\omega)$. In view of Proposition \ref{prop Holder estimate}, we know that these functions satisfy \eqref{eq2 Holder bounds} for a common constant $K=K(\alpha_0,\gamma,\lambda)>0$ given explicitly by \eqref{eq Holder bound}. 
Let us set 
\[
w_y(x):=\inf\left\{ w\in\CC(\R)\,:\,w\in \superSol(\lambda)(\omega),\ w(y)=0\right\}
\qquad
\hbox{for all $x\in\R$.}
\]
The function $w_y$ is well defined, it satisfies \eqref{eq2 Holder bounds}  and $w_y(y)=0$.  As an infimum of viscosity supersolutions of \eqref{eq cellPDE}, it is itself a supersolution. Let us show that $w_y$ solves \eqref{claim visco elsewhere} in the viscosity sense. If this where not the case, there would exist a $C^2$ strict supertangent $\varphi$ to $w_y$ at a point $x_0\not=y$\,\footnote{i.e., $\varphi>w_y$ in $\R\setminus\{x_0\}$ and $\varphi(x_0)=w_y(x_0)$.} such that \ $a(x_0,\omega)\varphi''(x_0)+H(x_0,\varphi'(x_0),\omega)<\lambda$. By continuity, we can pick $r>0$ small enough so that $|y-x_0|>r$ and 
\begin{equation}\label{eq failed test}
a(x,\omega)\varphi''(x)+H(x,\varphi'(x),\omega)<\lambda
\qquad
\hbox{for all $x\in (x_0-r,x_0+r)$.}
\end{equation}
Choose $\delta>0$ small enough so that 
\begin{equation}\label{eq boundary data}
\varphi(x_0+r)-\delta>w_y(x_0+r)\qquad\hbox{and}\qquad \varphi(x_0-r)-\delta>w_y(x_0-r).
\end{equation}
Let us set \ $\tilde w(x):=\min\{\varphi(x)-\delta,w_y(x)\}$ \ for all $x\in\R$. 
The function $\tilde w$ is the minimum of two viscosity supersolution of \eqref{eq cellPDE} in $(x_0-r,x_0+r)$, in view of \eqref{eq failed test}, and it agrees with $w_y$ in $\R\setminus [x_0-\rho,x_0+\rho]$ for a suitable $0<\rho<r$, in view of \eqref{eq boundary data}. We infer that $\tilde w\in\superSol$ with $\tilde w(y)=0$. But this contradicts the minimality of $w_y$ since $\tilde w(y)=\varphi(y)-\delta<w_y(y)$. 
\end{proof}

With the aid of Lemma \ref{lemma Busemann function}, we can now prove the following existence result. 

\begin{theorem}\label{teo existence solutions}
Let $\omega\in\Omega$ and $\lambda\geqslant \lambda_0(\omega)$ be fixed. Then there exists a viscosity solution $u$ to 
\eqref{eq cellPDE}. Furthermore $u$ is $K$-Lipschitz, where the constant  $K=K(\alpha_0,\alpha_1,\gamma,\kappa,\lambda)>0$ is given explicitly by \eqref{eq Lipschitz bound}, and of class $C^2$ on every interval $I$ contained in $\{a(\cdot,\omega)>0\}$. Furthermore, $\lambda_0(\omega)\geqslant \min H\geqslant -{1}/{\alpha_0}$. 

\begin{proof}
Let us first assume $\lambda>\lambda_0(\omega)$. 
Let us pick a sequence of points $(y_n)_n$ in $\R$ with $\lim_n |y_n|=+\infty$ and, for each $n\in\N$, set $w_n(\cdot):=w_{y_n}(\cdot)-w_{y_n}(0)$ on $\R$, where $w_{y_n}$ is the function provided by Lemma \ref{lemma Busemann function}. Accordingly, we know that the functions $w_n$ are equi-H\"older continuous and hence locally equi--bounded on $\R$ since $w_n(0)=0$ for all $n\in\N$. By the Ascoli-Arzel\`a Theorem, up to extracting a subsequence, the functions $w_n$ converge in $\CC(\R)$ to a limit function $u$. For every fixed bounded interval $I$, the functions $w_n$ are viscosity solutions of \eqref{eq cellPDE} for all $n\in\N$ big enough since $|y_n|\to +\infty$ as $n\to+\infty$, and so is $u$ by stability of the notion of viscosity solution. We can now apply Proposition \ref{prop regularity solutions} to infer that $u$ is $K$--Lipschitz continuous in $\R$ and of class $C^2$  on every interval $I$ contained in $A(\omega):=\{a(\cdot,\omega)>0\}$. If $\{a(\cdot,\omega)=0\}$ has positive Lebesgue measure, the function $u$ possesses a differentiability point $x_0\in \{a(\cdot,\omega)=0\}$, hence by Proposition \ref{prop pointwise solution} we get
\[
\lambda
= 
H(x_0,u'(x_0),\omega) 
\geqslant 
\min H.
\]
If instead  $\{a(\cdot,\omega)=0\}$ has zero Lebesgue measure, we have 
\[
\lambda-\min H
= 
a(x,\omega)u''+H(x,u',\omega) -\min H
\geqslant 
a(x,\omega)u''
\qquad
\hbox{for all $x\in A(\omega)=\{a(\cdot,\omega)>0\}$.}
\]
Being $u'$ essentially bounded on $\R$ and $A(\omega)$ a set of full measure in $\R$, we must have  $\lambda\geqslant \min H$. This readily implies $\lambda_0(\omega)\geqslant \min H\geqslant -1/\alpha_0$.  

Let us now choose a decreasing sequence of real numbers $(\lambda_n)_n$ converging to $\lambda_0(\omega)$ and, for each $n\in\N$, let $u_n$ be a viscosity solution of \eqref{eq cellPDE} with $\landa:=\landa_n$ and $u_n(0)=0$. In view of Proposition \ref{prop regularity solutions}, the functions $(u_n)_n$ are equi-Lipschitz and thus locally equi--bounded on $\R$, hence they converge, along a subsequence, to a function $u$ in $\CC(\R)$. By stability, $u$ is a solution of \eqref{eq cellPDE} with $\lambda:=\lambda_0(\omega)$. The Lipschitz and regularity properties of $u$ are again a consequence of Proposition \ref{prop regularity solutions}.
\end{proof}
\end{theorem}

The main output of the next result is that the function $\lambda_0(\cdot)$ is almost surely equal to a constant. We will denote this constant $\lambda_0$ in the sequel.

\begin{lemma}\label{lemma lambda zero}
The random variable $\lambda_0:\Omega\to \R$ is measurable and stationary. In particular, it is almost surely constant. 
\end{lemma}

\begin{proof}
If $u\in\CC(\R)$ is a viscosity supersolution of \eqref{eq cellPDE} for some $\omega\in\Omega$ and $\lambda\geqslant \lambda_0(\omega)$, then the functions $u(\cdot+z)$ is a viscosity supersolution of \eqref{eq cellPDE} with $\tau_z\omega$ in place of $\omega$ by the stationarity of $H$ and $a$. By its very definition, we infer that $\lambda_0$ is a stationary function. 

Let us show that  $\lambda_0:\Omega\to \R$ is measurable. Since the probability measure $\P$ is complete on $(\Omega,\F)$, it is enough to show that, for every fixed $\eps>0$, there exists a set $F\in \F$ with $\P(\Omega\setminus F)<\eps$ such that the restriction of 
$\lambda_0$ to $F$ is measurable.  To this aim, we notice that the measure $\P$ is inner regular on $(\Omega,\F)$, see \cite[Theorem 1.3]{Bill99}, hence it is a Radon measure. By applying Lusin's Theorem \cite{LusinThm} to the random variables $a:\Omega\to\CC(\R)$ and $H:\Omega\to\CC(\R\times\R)$, we infer that there exists a closed set $F\subseteq \Omega$ with $\P(\Omega\setminus F)<\eps$ such that $a_{| F}:F\to\CC(\R)$ and 
$H_{| F}:F\to\CC(\R\times\R)$ are continuous. 
We claim that $F\ni\omega\mapsto \lambda_0(\omega)\in \R$ is lower semicontinuous. 
Indeed, let $(\omega_n)_{n\in\N}$ be a sequence converging to some $\omega_0$ in $F$. For each $n\in\N$, let $u_n$ be a solution of \eqref{eq cellPDE} with $\omega_n$ and $\lambda(\omega_n)$ in place of $\omega$ and $\lambda$. Let us furthermore assume that $u_n(0)=0$ for all $n\in\N$. From Proposition \ref{prop regularity solutions} and the fact that $\landa_0(\omega)$ is uniformly bounded with respect to $\omega$, we derive that the functions $u_n$ are equi-Lipschitz and locally equi-bounded in $\R$. Let us extract a subsequence such that $\liminf_n \lambda (\omega_n)=\lim_k \lambda(\omega_{n_k})=:\tilde\lambda$ and 
$\big(u_{n_k}\big)_k$ converges to a function $u$ in $C(\R)$. Since $a(\cdot,\omega_n)\to a(\cdot,\omega_0)$ in 
$\CC(\R)$ and $H(\cdot,\cdot,\omega_n)\to H(\cdot,\cdot,\omega_0)$ in $\CC(\R\times\R)$, we derive by stability that $u$ solves \eqref{eq:cellPDE} with $\omega:=\omega_0$ and $\lambda:=\tilde\lambda$ in the viscosity sense. By definition of $\lambda_0(\omega_0)$, we conclude that $\tilde\lambda\geqslant \lambda(\omega_0)$, i.e., $\liminf_n \lambda(\omega_n)\geqslant \lambda_0(\omega_0)$ as it was claimed. 
\end{proof}

We conclude this section with the following remark.  

\begin{lemma}\label{lemma A(omega)}
There exists a set $\hat\Omega$ of probability 1 such that, for every $\omega\in\hat\Omega$, the following holds:
\begin{itemize}
\item[(A1$'$)] $a(\cdot,\omega)\not\equiv 0$ and every connected component of $A(\omega):=\{a(\cdot,\omega)>0\}$ is a bounded interval.
\end{itemize}
\end{lemma}

\begin{proof}
The set $C(\omega):=\{a(\cdot,\omega)=0\}$  is a stationary closed random set by \cite[Proposition 3.2]{DS09}. In view of \cite[Proposition 3.5]{DS09}, 
we infer that there exists a set $\hat\Omega$ of probability 1 such that, for every $\omega\in\hat\Omega$, we have 
$(-\infty,-n)\cap C(\omega)\not=\emptyset$ and $(n,+\infty)\cap C(\omega)\not=\emptyset$ for every $n\in\N$. This and condition (A1) readily imply the assertion. 
\end{proof}

\section{Deterministic solutions}\label{sec:deterministic  correctors}

\subsection{Assumptions and preliminary results}\label{sec:deterministic preliminaries}
In this section, we shall assume that $a:\R\to [0,1]$ is a deterministic diffusion coefficient satisfying condition (A2) and  
\begin{itemize}
\item[(A1$'$)]\quad $a\not\equiv 0$ and every connected component of the set $A:=\{a>0\}$ is a bounded interval.
\end{itemize}
We observe that, since $A$ is open and $\R$ is locally connected, every connected component of $A$ is always an open interval. As for the Hamiltonian $H:\R\times\R\to\R$, we will assume that it belongs to the class $\Hamsqc$ for some constants $\alpha_0,\alpha_1>0, \gamma>2$ and $\eta>0$.  In view of hypotheses (H2)-(H3), it is easily see that the Hamiltonian $H$ satisfies the following property:
\begin{itemize}
\item[(C)] for any $K>0$, there exists a constant $C(K)$ such that  $H$ is $C(K)$-Lipschitz in $\R\times [-K,K]$.
\end{itemize}
Note that $C(K)$ also depends on the constants $\alpha_1>0$ and $\gamma>2$. \smallskip

Throughout this section, we will adopt the following notation: for every $x\in\R$, we will denote by $\hat p(x)$ the unique element in $\R$ such that 
\begin{equation}\label{def p hat}
H(x,\hat p(x))=\min_{p\in\R} H(x,p)=:\hat\lambda(x);
\tag{$\hat p$}
\end{equation}
for every fixed $\lambda\geqslant \hat\lambda(x)$, we will denote by $p_\lambda^-(x),\, p_\lambda^+(x)$ the unique elements in $\R$ such that 
\begin{equation}\label{def p lambda}
\{p\in\R\,:\, H(x,p)\leqslant \lambda\,\}
=
[p_\lambda^-(x),p_\lambda^+(x)].
\tag{$p^\pm_\lambda$}
\end{equation}
We will denote by $\lambda_0$ the critical constant defined by \eqref{def lambda zero}. We recall that, by Theorem \ref{teo existence solutions}, the following viscous HJ equation 
\begin{equation}\label{appC eq cellPDE}
	a(x)u''+H(x,u')=\lambda\qquad\hbox{in $\R$}\tag{HJ$_\lambda$}
\end{equation}
admits Lipschitz viscosity solutions on $\R$ for every $\landa\geqslant\landa_0$.  \smallskip

We start by remarking the following fact. 

\begin{lemma}\label{appC lemma 1}
The functions $\hat \lambda,\hat p:\R\to\R$ are bounded and $\hat\kappa$-Lipschitz, for some constant $\hat\kappa=\hat\kappa(\alpha_0,\alpha_1,\gamma,\eta)>0$ only depending on $\alpha_0,\alpha_1>0,\gamma>2$ and $\eta>0$. 
\end{lemma}

\begin{proof}
According to  Lemma \ref{appA lemma Lipschitz min H}, we have $\sup_{x\in\R}|\hat p(x)|\leqslant \hat R$ 
and $\hat\landa(\cdot)$ is $\hat\kappa_1$-Lipschitz for some constants  $\hat R=\hat R(\alpha_0,\alpha_1,\gamma)$ and 
$\hat\kappa_1=\hat\kappa_1(\alpha_0,\alpha_1,\gamma)>0$ only depending on $\alpha_0,\alpha_1>0$ and $\gamma>2$. 
Let us denote by $\hat C=\hat C(\alpha_1,\gamma,\hat R)$ a Lipschitz constant of $H$ in $\R\times [-\hat R,\hat R]$. 
For every  $x,y\in\R$ we have
\[
H(x,\hat p(x))-H(y,\hat p(y))
\geqslant
-\hat C|x-y|+H(y,\hat p(x))-H(y,\hat p(y))
\geqslant
-\hat C|x-y|+\eta|\hat p(x)-\hat p(y)|,
\]
i.e.
\[
\eta|\hat p(x)-\hat p(y)|
\leqslant
\hat C|x-y|
+
\hat\landa(x)-\hat\landa(y)
\leqslant 
(\hat C+\hat\kappa_1) |x-y|.
\]
The assertion follows by setting $\hat\kappa:=\max\{\hat\kappa_1,(\hat C+\hat\kappa_1)/\eta\}$. 
\end{proof}

We state for further use two trivial consequences of the previous lemma. 

\begin{lemma}\label{appC lemma 2}
Let $x_0\in\{a=0\}$ and assume that $\lambda\geqslant \delta+\hat\lambda(x_0)$ for some $\delta>0$. For every $C>0$, there exists  $r=r(\hat\kappa,C,\delta)>0$ such that 
\[
\hat \lambda(x)<\lambda-3rC\qquad\hbox{for all $|x-x_0|<r$.}
\]
\end{lemma}

\begin{lemma}\label{appC lemma 3}
Let $x_0\in\{a=0\}$ and assume that $\lambda=\hat\lambda(x_0)$. For every $\eps>0$, there exists  $r=r(\hat\kappa,\eps)>0$ such that 
\[
\hat \lambda(x)<\lambda+\eps\qquad\hbox{for all $|x-x_0|<r$.}
\]
\end{lemma}

We proceed by showing the following preliminary existence result. 

\begin{prop}\label{appC prop pre-existence}
Let $\lambda> \lambda_0$ and let $u_{\landa_0}$ be a Lipchitz viscosity solution of \eqref{appC eq cellPDE} with $	\lambda=\landa_0$. Then there exists functions $f^-_\lambda, f^+_\lambda\in\CC^1(A)$ satisfying \ $f^-_\lambda< u'_{\lambda_0}< f^+_\lambda$\  in $A$\  and 
\begin{equation}\label{appC eq cellODE}
a(x)f'+H(x,f)=\lambda\qquad\hbox{in $A$.}\tag{ODE$_\lambda$}
\end{equation}
Furthermore, there exists a constant $K_\lambda=K(\alpha_0,\alpha_1,\gamma,\lambda)$, only depending on $\alpha_0,\alpha_1>0$, $\gamma>2$ and $\lambda$, such that 
\ $|f^\pm_\lambda(x)|\leqslant K_\lambda$\  for all $x\in A$.
\end{prop}

\begin{proof}
Let us prove the existence of such a function $f^+_\lambda$. In view of Proposition \ref{prop regularity solutions}, there exists a constant $K_{\lambda_0}=K(\alpha_0,\alpha_1,\gamma,\lambda_0)$, only depending on $\alpha_0,\alpha_1>0$, $\gamma>2$ and $\lambda_0$, such that 
$|u'_{\lambda_0}|\leqslant K_{\lambda_0}$ in $A$. Since $H\in\Ham$, we can find $q_\lambda>K_{\lambda_0}$ such that $H(x,q_\lambda)>\lambda$ for every $x\in\R$. Note that the choice of $q_\lambda$ only depends on  $\alpha_0,\alpha_1>0$, $\gamma>2$ and $\lambda$. 
Then the functions $m(\cdot):=u'_{\lambda_0}(\cdot)$ and $M(\cdot):=q_\lambda$ satisfy 
\begin{equation}\label{eq appC pre-existence}
a(x)m'(x)+H(x,m(x))<\lambda<a(x)M'(x)+H(x,M(x))\qquad\hbox{in $J$,}
\end{equation}
for every connected component $J$ of $A$. The existence of a $C^1$ solution $f^+_\lambda\in C^1(J)$ to \eqref{appC eq cellODE} in $J$ satisfying 
\[
u'_{\lambda_0}(x)=m(x)<f^+_\lambda(x)<q_\lambda=M(x)\qquad\hbox{for all $x\in J$}
\]
follows from standard ODE results, see for instance \cite[Lemma A.8]{DKY23} for a proof. The assertion follows by the arbitrariness of the choice of $J$. The proof for $f^-_{\lambda}$ is analogous and is omitted. 
\end{proof}

\subsection{Behavior of solutions near the boundary of $A$.}\label{sec:technical lemmas} A viscosity solution $u_\landa$ of \eqref{appC eq cellPDE} is of class $C^2$ in $A$ by Proposition \ref{prop regularity solutions}, in particular its derivative $f:=u'_\landa$ is a $C^1$ solution of \eqref{appC eq cellODE}. We need to understand the behavior of the derivative of  $u_\landa$ when we approach the boundary of $A$.  
In this subsection we study this issue when $f$ is a generic bounded $C^1$ solution of \eqref{appC eq cellODE} in $A$. The analysis is based on a series of technical propositions. We start with the following. 

\begin{prop}\label{appC prop tool1}
Let $x_0\in\{a=0\}$ and $\lambda\geqslant \delta+\hat\lambda(x_0)$ for some $\delta>0$. Let $f\in\CC^1(A)$ be a bounded solution of  \eqref{appC eq cellODE}. Let $C=C(K)$ the constant chosen according to (C) above with $K\geqslant \|f\|_{L^\infty(A)}$,  and $r(\hat\kappa,C,\delta)>0$ be the radius given by Lemma \ref{appC lemma 2}. Let  $r\in\big(0,r(\hat\kappa,C,\delta)\big)$ and 
$I=(\tilde\ell_1,\tilde\ell_2)$ be an interval contained in $(x_0-r,x_0+r)$ such that $a>0$ in $\overline I$. Assume the following hypotheses: 
\begin{itemize}
\item[\em(i)] \quad $f$ is monotone on $I$;\smallskip
\item[\em(ii)] \quad $f(\tilde\ell_1)>\hat p(\tilde\ell_1)$;\smallskip
\item[\em(iii)] \quad $\lambda-hC \leqslant H(\tilde\ell_1,f(\tilde\ell_1)) \leqslant \lambda+hC$\qquad where $h:=|I|\leqslant r$. 
 \smallskip
\end{itemize}
Then the following holds:
\begin{eqnarray*}
f(y)>\hat p(x)\qquad \hbox{for all $x,y\in I$;}\\
\lambda-2hC \leqslant H(x,f(x))\leqslant \lambda+2hC \qquad \hbox{for all $x\in I$}.
\end{eqnarray*}
\end{prop}

\begin{proof}
{\bf Case $f'\geqslant 0$ in $I$.} Then 
\begin{equation}\label{appC eq1.1 tool1}
H(x,f(x))\leqslant a(x)f'(x)+H(x,f(x))=\lambda\qquad\hbox{for all $x\in I$.}
\end{equation}
Since $f$ is nondecreasing on $I$, we have $f(x)\geqslant f(\tilde\ell_1)>\hat p(\tilde\ell_1)$ for all $x\in I$. Hence by (sqC) we infer
\begin{equation}\label{appC eq1.2 tool1}
H(x,f(x))\geqslant H(\tilde\ell_1,f(x))-C|\tilde\ell_1-x| \geqslant H(\tilde\ell_1,f(\tilde\ell_1))-hC\geqslant \lambda -2hC
\qquad 
\hbox{for all $x\in I$.} 
\end{equation}
From \eqref{appC eq1.1 tool1} and \eqref{appC eq1.2 tool1} we get 
\begin{equation*}
\lambda-2hC \leqslant H(x,f(x))\leqslant \lambda\qquad\hbox{for all $x\in I$,}
\end{equation*}
yielding in particular the second part of the assertion. 
Since $f(\tilde\ell_1)>\hat p(\tilde\ell_1)$, from \eqref{appC eq1.2 tool1} and the fact that 
$\lambda-3hC>\hat\lambda(x)$ for all $x\in I$ in view of the choice of $r$, we infer that \ $f(x)>\hat p(x)$ for all $x\in I$. Furthermore 
\begin{equation*}
H(y,f(x)) \geqslant H(x,f(x))-C|x-y|\geqslant \lambda-3hC>\hat\lambda(y)\qquad\hbox{for all $x,y\in I$.}
\end{equation*}
From the fact that $f(x)>\hat p(y)$ at $x=y$, we get from the continuity of $f$ that $f(x)>\hat p(y)$ for all $x,y\in I$, as it was to be shown.\smallskip\\
{\bf Case $f'\leqslant 0$ in $I$.} Then 
\begin{equation}\label{appC eq2.1 tool1}
H(x,f(x))\geqslant a(x)f'(x)+H(x,f(x))=\lambda>\hat\lambda(x)\qquad\hbox{for all $x\in I$.}
\end{equation}
From the fact that $f(\ell_1)>\hat p(\ell_1)$ we derive, by continuity of $f$ and $\hat p$, that $f(x)>\hat p(x)$ for all $x\in I$. 
For every $x<y$ in $I$ we have 
\begin{equation*}
H(x,f(x))
\geqslant
H(y,f(x)) - C|x-y|
\geqslant
H(y,f(y))-hC
\geqslant
\lambda -hC,
\end{equation*}
where for the second inequality we have used the fact that $f(x)\geqslant f(y)>\hat p(y)$ and the quasiconvexity of $H(y,\cdot)$.
For every $x,y\in I$ we infer 
\begin{equation*}
H(y,f(x))
\geqslant 
H(x,f(x))-C|x-y|
\geqslant
\lambda -2hC>\hat\lambda(y).
\end{equation*}
In view of the fact that $f(x)>\hat p(y)$ at the point $x=y$, we infer from the continuity of $f$ that $f(x)>\hat p(y)$ for all $x,y\in I$, thus proving the first part of the assertion. Last, for every $x\in I$ we get
\[
H(x,f(x)) 
\leqslant  
H(\tilde\ell_1,f(x))+C|x-\tilde\ell_1|
\leqslant
H(\tilde\ell_1,f(\tilde\ell_1))+hC
\leqslant 
\lambda +2hC, 
\] 
where for the second inequality we have used the fact that $\hat p(\tilde\ell_1)<f(x)\leqslant f(\tilde\ell_1)$ and the quasiconvexity of $H(\tilde\ell_1,\cdot)$. This, together  with \eqref{appC eq2.1 tool1}, proves the second part of the assertion.
\end{proof}

We proceed to show a variant of the previous result. 

\begin{prop}\label{appC prop tool2}
Let $x_0\in\{a=0\}$ and $\lambda\geqslant \delta+\hat\lambda(x_0)$ for some $\delta>0$. Let $f\in\CC^1(A)$ be a bounded solution of  \eqref{appC eq cellODE}. Let $C=C(K)$ the constant chosen according to (C) above with $K\geqslant \|f\|_{L^\infty(A)}$,  
and $r(\hat\kappa,C,\delta)>0$ be the radius given by Lemma \ref{appC lemma 2}. Let  $r\in\big(0,r(\hat\kappa,C,\delta)\big)$ and 
$I=(\tilde\ell_1,\tilde\ell_2)$ be an interval contained in $(x_0-r,x_0+r)$ such that $a>0$ on $\overline I$. Assume the following hypotheses:
\begin{itemize}
\item[\em(i)] \quad $f$ is monotone on $I$;\smallskip
\item[\em(ii)] \quad $\lambda-hC < H(\tilde\ell_2,f(\tilde\ell_2))<\lambda+hC$\qquad where $h:=|I|\leqslant r$. 
 \smallskip
\end{itemize}
Then:
\begin{itemize}
\item[(a)] if $f'\leqslant 0$ in $I$ and either $f(\tilde\ell_1)<\hat p(\tilde\ell_1)$ or  $f(\tilde\ell_2)<\hat p(\tilde\ell_2)$, then
\begin{eqnarray*}
f(y)<\hat p(x)\qquad &&\hbox{for all $x,y\in I$;}\\
\lambda\leqslant  H(x,f(x))<\lambda+2hC \qquad &&\hbox{for all $x\in I$};
\end{eqnarray*}
\item[(b)] if $f'\geqslant 0$ in $I$, then 
\begin{equation*}
H(x,f(x))\leqslant \lambda\qquad\hbox{for all $x\in I$.}
\end{equation*} 
Furthermore, if $f(\tilde\ell_2)<\hat p(\ell_2)$, then 
\begin{eqnarray*}
f(y)<\hat p(x)\qquad &&\hbox{for all $x,y\in I$;}\\
\lambda-2hC\leqslant H(x,f(x))\leqslant \lambda \qquad &&\hbox{for all $x\in I$}.
\end{eqnarray*}
\end{itemize}
\end{prop}

\begin{proof}
\noindent (a) Let us assume $f'\leqslant 0$ in $I$. Then 
\begin{equation}\label{appC eq1.1 tool2}
H(x,f(x))\geqslant a(x)f'(x)+H(x,f(x))=\lambda>\hat\lambda(x)\qquad\hbox{for all $x\in I$.}
\end{equation}
From the fact that either $f(\ell_1)<\hat p(\ell_1)$ or $f(\ell_2)<\hat p(\ell_2)$ we derive, by continuity of $f$ and $\hat p$, that $f(x)<\hat p(x)$ for all $x\in I$. 
We also have, for every $x,y\in I$, 
\begin{equation*}
H(y,f(x))
\geqslant
H(x,f(x)) - C|x-y|
\geqslant
\lambda -hC>\hat\lambda(y),
\end{equation*}
where the last inequality holds by the choice of $r$. Since $f(x)<\hat p(y)$ at $x=y$, by continuity of $f$ we get \  $f(x)<\hat p(y)$ for every $x,y\in I$. This prove the first part of assertion (i).
For $x<y$ in $I$ we have 
\[
H(x,f(x)) 
\leqslant 
H(y,f(x))+C|x-y|
\leqslant 
H(y,f(y))+hC, 
\]
where for the last inequality we have used the fact that $f(y)\leqslant f(x) <\hat p(y)$ and the quasiconvexity of $H(y,\cdot)$. By sending $y\to\tilde\ell_2^-$ we get, in view of the hypothesis, 
\[
H(x,f(x))\leqslant \lambda+2hC\qquad\hbox{for all $x\in I$.}
\]
This, together with \eqref{appC eq1.1 tool2}, proves the second part of assertion (a). \smallskip\\
(b) Let us assume $f'\geqslant 0$ in $I$. Then 
\begin{equation}\label{appC eq2.1 tool2}
H(x,f(x))\leqslant a(x)f'(x)+H(x,f(x))=\lambda\qquad\hbox{for all $x\in I$,}
\end{equation}
thus proving the first inequality in assertion (b). If we assume $f(\tilde \ell_2)<\hat p(\tilde\ell_2)$, by monotonicity of $f$ we derive 
$f(x)\leqslant f(\tilde \ell_2)<\hat p(\tilde\ell_2)$ for all $x\in I$. We infer 
\begin{equation}\label{appC eq2.2 tool2}
H(x,f(x))
\geqslant
H(\tilde\ell_2,f(x))-C|x-\tilde\ell_2|
\geqslant
H(\tilde\ell_2,f(\tilde\ell_2))-hC
\geqslant 
\lambda-2hC
>
\hat\lambda(y)
\quad
\hbox{for all $x,y\in I$},
\end{equation}
where for the second inequality we have used the fact that $f(x)\leqslant f(\tilde\ell_2)<\hat p(\ell_2)$ and the quasiconvexity of $H(\tilde\ell_2,\cdot)$. For every fixed $x\in I$ we have $f(x)<\hat p(y)$ at $y=\tilde \ell_2$, 
hence, by continuity of $\hat p$, we derive from \eqref{appC eq2.2 tool2} that $f(x)<\hat p(y)$ for every $y\in I$. This, together with \eqref{appC eq2.1 tool2} and \eqref{appC eq2.2 tool2}, proves the second part of assertion (b).   
\end{proof}

The next proposition addresses the remaining case $\lambda=\hat\lambda(x_0)$.

\begin{prop}\label{appC prop tool3}
Let $x_0\in\{a=0\}$ and $\lambda=\hat\lambda(x_0)$. Let $f\in\CC^1(A)$ be a bounded solution of  \eqref{appC eq cellODE}. Let $C=C(K)$ the constant chosen according to (C) above with $K\geqslant \|f\|_{L^\infty(A)}$. Let $\eps>0$ and denote by $r(\hat\kappa,\eps)>0$ be the radius given by Lemma \ref{appC lemma 3}. Let  $r\in\big(0,r(\hat\kappa,\eps)\big)$ and 
$I=(\tilde\ell_1,\tilde\ell_2)$ be an interval contained in $(x_0-r,x_0+r)$ such that $a>0$ on $\overline I$. Assume the following hypotheses:
\begin{itemize}
\item[\em(i)] \quad $f$ is monotone on $I$;\smallskip
\item[\em(ii)] \quad $H(\tilde\ell_i,f(\tilde\ell_i))<\lambda+\eps$\qquad for each $i\in\{1,2\}$.  
 \smallskip
\end{itemize}
Then\ \  
%\begin{eqnarray*}
$H(x,f(x))<\lambda+\eps+2hC$ \  {for all $x\in I$, where $h:=|I|\leqslant r$.}
%\end{eqnarray*}
\end{prop}

\begin{proof}
We have 
\[
H(\tilde\ell_1,f(\tilde\ell_2))
\leqslant
H(\tilde\ell_2,f(\tilde\ell_2))+C|\tilde\ell_2-\tilde\ell_1|
<
\lambda+\eps+hC.
\]
Since $f$ is monotone on $I$,  for any fixed $x\in I$ the point $f(x)$ is contained in the interval having $f(\tilde\ell_1)$ and $f(\tilde\ell_2)$ as extreme points. For every $x\in I$ we get 
\[
H(x,f(x)) 
\leqslant
H(\tilde\ell_1,f(x))+C|x-\tilde\ell_1|
\leqslant
\max_{i\in\{1,2\}} \big\{ H(\tilde\ell_1, f(\tilde\ell_1)), H(\tilde\ell_1, f(\tilde\ell_2))\big\}
+hC
<
\lambda+\eps+2hC,
\]
where for the second inequality we have exploited the quasiconvexity of $H(\tilde\ell_1,\cdot)$. 
\end{proof}

We are now in position to establish the first boundary regularity result of this subsection.

\begin{theorem}\label{appC teo boundary regularity}
Let $J$ be an open and bounded interval such that $a>0$ on $J$ and let $x_0\in\partial J$ such that $a(x_0)=0$.  
Let $f\in\CC^1(J)$ be a bounded solution of  \eqref{appC eq cellODE} for some $\lambda$ in $\R$.  
Then the following limit exists and it satisfies
\[
\lim_{x\to x_0,\, x\in J}f(x)\in \{p_\lambda^-(x_0), p_\lambda^+(x_0)\}.
\]
In particular, $\lambda\geqslant \hat\lambda(x_0)$. 
\end{theorem}

\begin{proof}
Choose $K\geqslant  \|f\|_{L^\infty(J)}$ and denote by $C=C(K)$ the constant chosen according to (C) above. For every $y<x$ in $J$ we have 
\[
f(x)-f(y)=\int_y^x \dfrac{\lambda -H(z,f(z))}{a(z)}\, dz,
\]
i.e.
\begin{equation}\label{appC eq not exploding}
\left| \int_y^x \dfrac{\lambda -H(z,f(z))}{a(z)}\, dz\right|
\leqslant 2K
\qquad
\hbox{for all $y<x$ in $J$.}
\end{equation}
We first consider the case when $x_0$ is the right end point of the interval $J$. 
Let us assume there exists a $\rho>0$ such that $f$ is monotone in $J\cap (x_0-\rho,x_0)$. By monotonicity we infer that $f(x)$ has a limit $p$ as $x\to x_0^-$ in $J$. If $H(x_0,p)\not = \lambda$, there would exist $\rho>0$ small enough such that $H(x,f(x))-\lambda$ has constant sign in $J\cap (x_0-\rho,x_0)$ and $|H(x,f(x))-\lambda|>\eps$ for all $x\in J\cap (x_0-\rho,x_0)$ for some $\eps>0$. This contradicts \eqref{appC eq not exploding} since, for every fixed $y\in J$,  $\int_{y}^x \frac{1}{a(z)}\, dz\to +\infty$ as 
$x\to x_0^-$ by Lemma \ref{lemma 1/a}.
%\[
%\lim_{x\to x_0^-} \int_y^x \dfrac{1}{a(z)}\, dz =+\infty\qquad\hbox{for every fixed $y\in J$.}
%\]
This implies the assertion. 

Let us assume on the contrary that $f$ is not monotone in $J\cap (x_0-\rho,x_0)$ for every $\rho>0$. 
In particular, there exists an increasing sequence of points $(x_n)_n$ converging to $x_0$ such that $f'(x_n)=0$ for all $n\in\N$, hence 
\[
\lambda=a(x_n)f'(x_n)+H(x_n,f(x_n))=H(x_n,f(x_n))\qquad\hbox{for all $n\in\N$.}
\]
We derive that any accumulation point $p$ of $\big(f(x_n)\big)_n$ satisfies $H(x_0,p)=\lambda$, in particular $\lambda\geqslant \hat\lambda(x_0)$. We split the proof in two subcases.\smallskip\\ 
\noindent{\bf Case $\lambda>\hat\lambda(x_0)$.} Let us denote by $r(\hat\kappa,C,\lambda-\hat\lambda(x_0))$ the radius given by Lemma \ref{appC lemma 2} and choose $r\in\big(0,r(\hat\kappa,C,\lambda-\hat\lambda(x_0))\big)$. By the non monotonicity assumption on $f$ in a any left neighborhood of $x_0$, we know from what remarked above that there exist  infinitely many $x\in (x_0-r,x_0)$ such that $f'(x)=0$. For such an $x$, by arguing as above, we infer that $H(x,f(x))=\lambda>\hat\lambda(x)$, where the last inequality follows from the choice of $r$ in view of Lemma \ref{appC lemma 2}. 

Let us assume that there exists $\ell\in (x_0-r,x_0)$ such that $H(\ell,f(\ell))=\lambda$ and $f(\ell)>\hat p(\ell)$. Denote by $E$ the (nonempty) set of points $x\in [\ell,x_0)$ that satisfy the following inequalities:
\begin{itemize}
\item[(a)] \quad $f(x)>\hat p(x)$;\smallskip
\item[(b)] \quad $\lambda-2Cr \leqslant H(x,f(x)) \leqslant \lambda +2Cr$. 
\end{itemize}
Let us set $\overline y:=\sup\{y\in [\ell,x_0)\,:\, [\ell,y)\subseteq E\,\}$. By continuity of $f$, $\hat p$ and $H$, we certainly have $\overline y>\ell$. We claim that $\overline y=x_0$. Let us assume by contradiction that $\overline y\in (\ell,x_0)$. By continuity of $f$, $\hat p$ and $H$, we get that $\overline y$ satisfies condition (b) above and $f(\overline y)\geqslant \hat p(\overline y)$. Since 
$\lambda-2Cr>\hat\lambda(\overline y)$ by the choice of $r>0$, we infer that $\overline y$ satisfies (a) as well, i.e. $\overline y\in E$.  
If $f'(\overline y)=0$, then $H(\overline y, f(\overline y))=\lambda$ and by continuity of $f$, $\hat p$ and $H$ we can find $\rho>0$ small enough such that $(\overline y-\rho,\overline y+\rho)\subset E$, thus contradicting the maximality of $\overline y$. 
Let us assume, on the other hand, that $f'(\overline y)\not=0$, and take the connected component $I$ of the set 
$\{x\in (\ell,x_0)\,:\,f'(x)\not=0\,\}$ containing $\overline y$. Then $I$ is a bounded open interval of the form $I=(\tilde \ell_1,\tilde \ell_2)$ with  $\ell\leqslant \tilde\ell_1<\overline y<\tilde\ell_2< x_0$, where the strict inequality $\tilde \ell_2<x_0$ comes from the fact that we are assuming that $f$ is not monotone in any left neighbourhood of $x_0$.  In any case, no matter if $\tilde\ell_1>\ell_1$ or $\tilde\ell_1=\ell_1$, we have that $f'(\tilde\ell_1)=0$, hence $H(\tilde\ell_1,f(\tilde\ell_1))=\lambda$ and $f(\tilde\ell_1)>\hat p(\tilde\ell_1)$ by definition of $\overline y$. 
By applying Proposition \ref{appC prop tool1} we infer that $I\subset E$, and hence 
$[\ell,\tilde\ell_2]=[\ell,\overline y]\cup I \subseteq E$, thus contradicting the maximality of $\overline y$. Hence $\overline y=x_0$. This means that any accumulation point $p$ of the (bounded) sequence $\big(f(z_n)\big)_n$,  where $(z_n)_n$ is any fixed sequence of points converging to $x_0$ from the left, satisfies $p>\hat p(x_0)$ and $\lambda-2Cr \leqslant H(x,f(x)) \leqslant \lambda +2Cr$ for every $r>0$ small enough. This readily implies that $\lim_{x\to x_0^-} f(x)=p_\lambda^+(x_0)$.  

We are left with the case when $f(x)<\hat p(x)$ for every $x\in (x_0-r,x_0)$ such that $H(x,f(x))=\lambda$. Pick $\ell\in (x_0-r,x_0)$ such that $H(\ell,f(\ell))=\lambda$ and $f(\ell)<\hat p(\ell)$ and denote by $F$ the set of points $x\in [\ell,x_0)$ that satisfy the following inequalities:
\begin{itemize}
\item[(a$'$)] \quad $f(x)<\hat p(x)$;\smallskip
\item[(b$'$)] \quad $\lambda-2Cr \leqslant H(x,f(x)) \leqslant \lambda +2Cr$. 
\end{itemize}
Let us set $\overline y:=\sup\{y\in [\ell,x_0)\,:\, [\ell,y)\subseteq F\,\}$. We claim that $\overline y=x_0$. Let us assume by contradiction that $\overline y\in (\ell,x_0)$. 
The argument is similar to the one above. First, we get that $\overline y>\ell$ and 
$\overline y\in F$. If $f'(\overline y)=0$, by arguing as above we again reach a contradiction with the maximality of $\overline y$. 
Let us then assume that $f'(\overline y)\not=0$, and take the connected component $I$ of the set 
$\{x\in (\ell,x_0)\,:\,f'(x)\not=0\,\}$ containing $\overline y$. Then $I$ is a bounded open interval of the form $I=(\tilde \ell_1,\tilde \ell_2)$ with $\ell\leqslant \tilde\ell_1<\overline y<\tilde\ell_2< x_0$. 
%where the strict inequality $\tilde \ell_2<x_0$ comes from the fact that we are assuming that $f$ is not monotone in any left neighbourhood of $x_0$.  
Then $f'(\tilde\ell_2)=0$, hence $H(\tilde\ell_2,f(\tilde\ell_2))=\lambda$, which implies that $f(\tilde\ell_2)<\hat p(\tilde\ell_2)$ since this is the case we are considering.  
By applying Proposition \ref{appC prop tool2} we infer that $I\subset F$, and hence 
$[\ell,\tilde\ell_2]=[\ell,\overline y]\cup I \subseteq F$, thus contradicting the maximality of $\overline y$. By arguing as above, we conclude that $\lim_{x\to x_0^-} f(x)=p_\lambda^-(x_0)$.  
\smallskip\\
\noindent{\bf Case $\lambda=\hat\lambda(x_0)$.} Let us fix $\eps>0$ and denote by $r(\hat\kappa,\eps)>0$ 
the radius given by Lemma \ref{appC lemma 3}. Choose $r\in \big(0,r(\hat\kappa,\eps)\big)$ and pick $\ell\in (x_0-r,x_0)$ such that $f'(\ell)=0$ and hence $H(\ell,f(\ell))=\lambda$. Take $y\in (\ell,x_0)$. If $f'(y)=0$, then $H(y,f(y))=\lambda$. If $f'(y)\not=0$, take the connected component $I$ of the set $\{x\in (\ell,x_0)\,:\,f'(x)\not=0\,\}$ containing $y$. Then $I$ is a bounded open interval of the form $I=(\tilde \ell_1,\tilde \ell_2)$ with $\ell\leqslant \tilde\ell_1<\overline y<\tilde\ell_2< x_0$, where the strict inequality $\tilde \ell_2<x_0$ comes from the fact that we are assuming that $f$ is not monotone in any left neighbourhood of $x_0$. For each $i\in\{1,2\}$ we have 
$f'(\tilde\ell_i)=0$ and hence $H(\tilde\ell_i,f(\tilde\ell_i))=\lambda$. By applying Proposition \ref{appC prop tool3} we infer that 
$H(x,f(x))<\lambda+\eps+2Cr$ for every $x\in I$. By arbitrariness of the choice of $\eps>0$ and $r\in \big(0,r(\hat\kappa,\eps)\big)$, we conclude that $\lim_{x\to x_0^-} H(x,f(x))=\landa=\hat\lambda(x_0)$, namely 
$\lim_{x\to x_0^-} f(x)=\hat p(x_0)=p_\lambda^-(x_0)=p_\lambda^+(x_0)$.\smallskip

When $x_0$ is the left end point of $J=(x_0,\tilde\ell)$, we can reduce to the case considered above after remarking that the function $\check f(x):=-f(-x)$ is a solution to \eqref{appC eq cellODE} in $\check J:=(-\tilde\ell,-x_0)$ with $\check a(x):=a(-x)$ and $\check H(x,p):=H(-x,-p)$ in place of $a$ and $H$, respectively. 
\end{proof}

Next, we prove a refined version of Proposition \ref{appC prop tool2}, where we drop the monotonicity condition previously assumed on $f$.

\begin{prop}\label{appC prop tool2+}
Let $x_0\in\{a=0\}$ and $\lambda\geqslant \delta+\hat\lambda(x_0)$ for some $\delta>0$. Let $f\in\CC^1(A)$ be a bounded solution of  \eqref{appC eq cellODE}. Let $C=C(K)$ the constant chosen according to (C) above with $K\geqslant \|f\|_{L^\infty(A)}$,  
and $r(\hat\kappa,C,\delta)>0$ be the radius given by Lemma \ref{appC lemma 2}. Let  $r\in\big(0,r(\hat\kappa,C,\delta)\big)$ and 
$J=( \ell_1, \ell_2)$ be a connected component of $A$ contained in $(x_0-r,x_0+r)$. Then 
\begin{equation}\label{appC claim prop tool2+}
H(x,f(x))\leqslant \lambda +2hC\quad\hbox{for all $x\in J$, where $h:=|I|\leqslant r$.}
\end{equation}
Furthermore:
\begin{itemize}
\item[(a)] if $f(x)>\hat p(x)$ in a right neighborhood of $ \ell_1$, then 
\begin{eqnarray*}
f(y)>\hat p(x)\qquad &&\hbox{for all $x,y\in J$;}\\
\lambda-2hC \leqslant H(x,f(x)) \qquad &&\hbox{for all $x\in J$};\smallskip
\end{eqnarray*}
\item[(b)] if $f(x)<\hat p(x)$ in a left neighborhood of $ \ell_2$, then 
\begin{eqnarray*}
f(y)<\hat p(x)\qquad &&\hbox{for all $x,y\in J$;}\\
\lambda-2hC \leqslant H(x,f(x)) \qquad &&\hbox{for all $x\in J$}.
\end{eqnarray*}
\end{itemize}
\end{prop}

\begin{proof}
From the fact that $J=(\ell_1,\ell_2)$ is contained in $(x_0-r,x_0+r)$, we derive that $a(\ell_1)=a(\ell_2)=0$. We can therefore apply 
Theorem \ref{appC teo boundary regularity} to infer that there exists $0<\delta<|J|/2$ such that 
\begin{equation}\label{appC eq1 prop tool2+}
\hat\lambda(y)<\lambda-hC<H(x,f(x))<\lambda+hC\qquad\hbox{for every $x,y\in (\ell_1,\ell_1+\delta]\cup [\ell_2-\delta,\ell_2)$,}
\end{equation}
where the first inequality holds due to the choice of $r$. Hence we need to show \eqref{appC claim prop tool2+} only for points $y\in J_\delta:= (\ell_1+\delta,\ell_2-\delta)$. 
Let $y\in J_\delta$. If $f'(y)\geqslant 0$,  then 
\[
H(y,f(y))\leqslant a(y)f'(y)+H(y,f(y))=\lambda.
\]
Let us assume $f'(y)<0$ and denote by $I=(\tilde\ell_1,\tilde\ell_2)$ the connected component of the set 
$\{x\in J_\delta\,:\,f'(x)<0\,\}$ containing the point $y$. Note that  $\ell_1+\delta\leqslant \tilde\ell_1<y<\tilde\ell_2\leqslant \ell_2-\delta$. We remark that \eqref{appC eq1 prop tool2+} holds with $x:=\tilde\ell_i$ for each $i\in\{1,2\}$. This is trivially true due to \eqref{appC eq1 prop tool2+} if $|\tilde\ell_i-\ell_i|=\delta$. 
If on the other hand $|\tilde\ell_i-\ell_i|>\delta$, then $f'(\tilde\ell_i)=0$ and so $H(\tilde\ell_i,f(\tilde\ell_i))=\lambda$. If $f(\tilde\ell_1)>\hat p(\tilde\ell_1)$, we can apply Proposition \ref{appC prop tool1} and get
\[
\lambda-2hC \leqslant H(x,f(x))\leqslant \lambda+2hC \qquad \hbox{for all $x\in I$.}
\]
If $f(\tilde\ell_1)<\hat p(\tilde\ell_1)$, we apply Proposition \ref{appC prop tool2}-(a) to get 
\[
\lambda\leqslant  H(x,f(x))<\lambda+2hC \qquad \hbox{for all $x\in I$}.
\]
This proves the first assertion. 

We now proceed to prove assertion (a). If $f(x)>\hat p(x)$ in a right neighborhood of $\ell_1$, by the continuity of $f$ and $\hat p$ we get that the first two inequalities in \eqref{appC eq1 prop tool2+} imply that $f(x)>\hat p(x)$ for every $x\in (\ell_1,\ell_1+\delta]$. 
%If $f$ is monotone in $J_\delta$, the assertion follows as a direct application of Proposition \ref{appC prop tool1} with $I:=J_\delta$. Let us assume that $f$ is not monotone in $J_\delta$.  
%
Denote by $E$ the (nonempty) set of points $x\in [\ell_1+\delta,\ell_2)$ that satisfy the following inequalities:
\begin{itemize}
\item[(i)] \quad $f(x)>\hat p(x)$;\smallskip
\item[(ii)] \quad $\lambda-2hC \leqslant H(x,f(x)) \leqslant \lambda +2hC$. 
\end{itemize}
Let us set $\overline y:=\sup\{y\in [\ell_1+\delta,\ell_2)\,:\, [\ell_1+\delta,y)\subseteq E\,\}$. By continuity of $f$, $\hat p$ and $H$, we certainly have $\overline y>\ell_1+\delta$. We claim that $\overline y=\ell_2$. Let us assume by contradiction that $\overline y\in (\ell_1+\delta,\ell_2)$. By continuity of $f$, $\hat p$ and $H$, we get that $\overline y$ satisfies condition (ii) above and $f(\overline y)\geqslant \hat p(\overline y)$. Since 
$\lambda-2hC>\hat\lambda(\overline y)$ by the choice of $r>0$, we infer that $\overline y$ satisfies (i) as well, i.e. $\overline y\in E$.  
If $f'(\overline y)=0$, then $H(\overline y, f(\overline y))=\lambda$ and by continuity of $f$, $\hat p$ and $H$ we can find $\rho>0$ small enough such that $(\overline y-\rho,\overline y+\rho)\subset E$, thus contradicting the maximality of $\overline y$. 
Let us assume, on the other hand, that $f'(\overline y)\not=0$, and take the connected component $I$ of the set 
$\{x\in (\ell_1+\delta,\ell_2)\,:\,f'(x)\not=0\,\}$ containing $\overline y$. Then $I$ is a bounded open interval of the form $I=(\tilde \ell_1,\tilde \ell_2)$ with  $\ell_1+\delta\leqslant \tilde\ell_1<\overline y<\tilde\ell_2\leqslant \ell_2$.  By definition of $\overline y$, we have that the point $\tilde\ell_1$ satisfies condition (i) above. Furthermore, if $\tilde\ell_1=\ell_1+\delta$, then the point $x:=\tilde\ell_1$ satisfies \eqref{appC eq1 prop tool2+}. If $\tilde\ell_1>\ell_1+\delta$, then $f'(\tilde\ell_1)=0$ and so $H(\tilde\ell_1,f(\tilde\ell_1))=\lambda$. In either case, we can apply Proposition \ref{appC prop tool1} to derive that $I\subset E$, and hence 
$[\ell_1+\delta,\tilde\ell_2]=[\ell_1+\delta,\overline y]\cup I \subseteq E$, thus contradicting the maximality of $\overline y$. Hence $\overline y=\ell_2$. We have thus proved that the inequalities (i) and (ii)  
%\eqref{appC eq1 prop tool2+} and $f(x)>\hat p(x)$ 
hold for every 
$x\in J$. Since for every fixed $y\in J$ the inequality $f(x)>\hat p(y)$ is satisfied at $x=y$, from (ii), the choice of $r>0$ and 
%\eqref{appC eq1 prop tool2+} and  
the fact that $f$ is continuous we derive that $f(x)>\hat p(y)$ for every $x,y\in J$. 

Assertion (b) follows directly from item (a) after observing that the function $\check f(x):=-f(-x)$ is a solution to \eqref{appC eq cellODE} in $\check J:=(-\ell_2,-\ell_1)$ with $\check a(x):=a(-x)$ and $\check H(x,p):=H(-x,-p)$ in place of $a$ and $H$, respectively. 
\end{proof}

We will also need the following refined version of Proposition \ref{appC prop tool3}, where we drop the monotonicity condition previously assumed on $f$.

\begin{prop}\label{appC prop tool3+}
Let $x_0\in\{a=0\}$ and $\lambda=\hat\lambda(x_0)$. Let $f\in\CC^1(A)$ be a bounded solution of   \eqref{appC eq cellODE}. Let $C=C(K)$ the constant chosen according to (C) above with $K\geqslant \|f\|_{L^\infty(A)}$. Let $\eps>0$ and denote by $r(\hat\kappa,\eps)>0$ be the radius given by Lemma \ref{appC lemma 3}. Let  $r\in\big(0,r(\hat\kappa,\eps)\big)$ and 
$J=( \ell_1, \ell_2)$ be a connected component of the set $A$ contained in $(x_0-r,x_0+r)$. Then\ \  
\begin{equation}\label{appC claim prop tool3+}
H(x,f(x))<\lambda+\eps+2hC\qquad\hbox{for all $x\in J$, where $h:=|J|\leqslant r$.}
\end{equation}
\end{prop}

\begin{proof}
From the fact that $J=(\ell_1,\ell_2)$ is contained in $(x_0-r,x_0+r)$, we derive that $a(\ell_1)=a(\ell_2)=0$. We can therefore apply  Theorem \ref{appC teo boundary regularity}, there exists $0<\delta<|J|/2$ such that 
\begin{equation}\label{appC eq1 prop tool3+}
H(x,f(x))<\lambda+\eps\qquad\hbox{for every $x,y\in (\ell_1,\ell_1+\delta]\cup [\ell_2-\delta,\ell_2)$.}
\end{equation}
Hence we need to show \eqref{appC claim prop tool3+} only for points $y\in J_\delta:= (\ell_1+\delta,\ell_2-\delta)$. 
Let $y\in J_\delta$.  If $f'(y)\geqslant 0$,  then 
\[
H(y,f(y))\leqslant a(y)f'(y)+H(y,f(y))=\lambda.
\]
Let us assume $f'(y)<0$ and denote by $I=(\tilde\ell_1,\tilde\ell_2)$ the connected component of the set 
$\{x\in J_\delta\,:\,f'(x)<0\,\}$ containing the point $y$. Note that  $\ell_1+\delta\leqslant \tilde\ell_1<y<\tilde\ell_2\leqslant \ell_2-\delta$. We remark that \eqref{appC eq1 prop tool3+} holds with $x:=\tilde\ell_i$ for each $i\in\{1,2\}$. This is trivially true due to \eqref{appC eq1 prop tool3+} if $|\tilde\ell_i-\ell_i|=\delta$. 
If on the other hand $|\tilde\ell_i-\ell_i|>\delta$, then $f'(\tilde\ell_i)=0$ and so $H(\tilde\ell_i,f(\tilde\ell_i))=\lambda$. We can therefore apply Proposition \ref{appC prop tool3} and conclude that \eqref{appC claim prop tool3+} holds for every $x\in I$ and in particular for $x=y$. 
\end{proof}

\subsection{Main results}\label{sec:deterministic correctors main} We take advantage of the information gathered in the previous subsection and proceed to study the regularity and uniqueness properties of the viscosity solutions to \eqref{appC eq cellPDE} we are interested in. We start with the following crucial result. 

\begin{theorem}\label{appC teo boundary regularity2}
Let $x_0\in \{a=0\}$ be an accumulation point for $A$. Let 
$\lambda>\lambda_0$ and  $f\in\CC^1(A)$ be a bounded solution of   \eqref{appC eq cellODE}. 
Let $u_{\landa_0}$ be a Lipchitz viscosity solution of \eqref{appC eq cellPDE} with $	\lambda=\landa_0$. 
\begin{itemize}
\item[(a)] If $f\geqslant u'_{\lambda_0}$ on $A$, then \ \ 
$\displaystyle \lim_{x\to x_0,\, x\in A}f(x)=p_\lambda^+(x_0)>p_{\lambda_0}^+(x_0)\geqslant \hat p(x_0).$
\item[(b)] If $f\leqslant  u'_{\lambda_0}$ on $A$,  then \ \ 
$\displaystyle \lim_{x\to x_0,\, x\in A}f(x)=p_\lambda^-(x_0)<p_{\lambda_0}^-(x_0)\leqslant \hat p(x_0).$
\end{itemize}
In particular, $\hat\lambda(x_0)\leqslant \landa_0$.
\end{theorem}

\begin{proof}
Let us set $f_{\lambda_0}:=u'_{\lambda_0}$. In view of Proposition \ref{prop regularity solutions}, the function $f_{\lambda_0}$ is a $C^1$ solution of \eqref{appC eq cellODE} with $\lambda:=\lambda_0$.

We only consider the case when $x\to x_0^-$. The other case can be treated analogously. 
If there exists $\rho>0$ such that $(x_0-\rho,x_0)\subset A$, then, according to Theorem \ref{appC teo boundary regularity}, there exists $p_\lambda\in\{p^-_\lambda(x_0), p^+_\lambda(x_0)\}$ and 
$p_{\lambda_0}\in\{p^-_{\lambda_0}(x_0), p^+_{\lambda_0}(x_0)\}$ such that $f_\lambda(x)\to p_\lambda$ and $f_{\lambda_0}(x)\to p_{\lambda_0}$ as $x\to x_0^-$. In case (a), we have 
$p_\lambda\geqslant p_{\lambda_0} \geqslant p^-_{\lambda_0}(x_0)$, hence 
$p_\lambda=p^+_\lambda(x_0)>p_{\lambda_0}^+(x_0)\geqslant \hat p(x_0)$ by  quasiconvexity of $H(x_0,\cdot)$. 
 In case (b), we have 
$p_\lambda\leqslant p_{\lambda_0} \leqslant p^+_{\lambda_0}(x_0)$, hence 
$p_\lambda=p^-_\lambda(x_0)<p_{\lambda_0}^-(x_0)\leqslant \hat p(x_0)$ by quasiconvexity of $H(x_0,\cdot)$. 

If $(x_0-\rho,x_0)\not\subset A$ for every $\rho>0$, then there exist a sequence of connected component 
$I_n:=(\underline \ell_n, \overline \ell_n)$ of $A$ such that $\underline\ell_n\to x_0^-$ as $n\to +\infty$. 
From the first part of the proof applied to each interval $I_n$, we infer that $\hat\landa(\ul\ell_n)\leqslant \landa_0<\landa$. 
By continuity of $\hat\landa(\cdot)$, we get  $\hat\lambda(x_0)\leqslant \landa_0<\lambda$.

Let $C=C(K)$ the constant chosen according to (C) above with $K\geqslant \|f\|_{L^\infty(A)}$,  
and $r(\hat\kappa,C,\delta)>0$ be the radius given by Lemma \ref{appC lemma 2} with 
$\delta:=\lambda-\lambda_0$. Fix $r\in\big(0,r(\hat\kappa,C,\delta)\big)$ and let $I=(\ell_1,\ell_2)$ be a connected component of $A$ contained in $(x_0-r,x_0+r)$.  If we are in case (a), then, according to the first part of the proof, we have that $f(x)>\hat p(x)$  
in a right neighborhood of $\ell_1$. In view of Proposition \ref{appC prop tool2+}-(a) we get 
\[
f(x)>\hat p(x)\quad\hbox{and}\quad \lambda-2rC \leqslant H(x,f(x))\leqslant \lambda+2rC \qquad \hbox{for all $x\in I$.}
\]
This readily implies assertion (a). If we are in case (b), then, according to the first part of the proof, we have that $f(x)<\hat p(x)$ in a left neighborhood of $\ell_2$. We can therefore apply Proposition \ref{appC prop tool2+}-(b) and conclude via an analogous argument. 
\end{proof}

As a consequence of the above, we derive the following important information. 

\begin{prop}\label{appC prop lambda 0}
We have $\lambda_0\geqslant \sup\limits_{x\in\{a=0\}} \hat\lambda(x)$.
\end{prop}

\begin{proof}
According to Proposition \ref{appC prop pre-existence},  for every $\lambda>\lambda_0$ there exists a bounded function $f\in\CC^1(A)$ which solves  \eqref{appC eq cellODE} and satisfies $f\geqslant u'_{\lambda_0}$ on $A$.  
If $x_0\in \{a=0\}$ is an accumulation point for $A$, then, according to Theorem \ref{appC teo boundary regularity2}, $\hat \lambda(x_0)\leqslant \landa_0$. If $x_0\in \{a=0\}$ is not an accumulation point for $A$, then there exists $\rho>0$ such that $(x_0-\rho,x_0+\rho)\subset \{a=0\}$. From the fact that $H(x,u_{\lambda_0}'(x))= \lambda_0$ for a.e. $x\in\{a=0\}$ in view of Proposition \ref{prop pointwise solution}, we obtain in particular that 
$\hat\lambda(x)\leqslant \lambda_0$ for a.e. $x\in (x_0-\rho,x_0+\rho)$, and hence for every $x\in (x_0-\rho,x_0+\rho)$ by continuity of $\hat\lambda(\cdot)$. 
\end{proof}

We proceed by showing the following fact, see \eqref{def p lambda} for the definition of $p_\lambda^-,\, p_\lambda^+$. 

\begin{prop}\label{appC prop p lambda}
Let $\lambda\geqslant \sup_{x\in\{a=0\}} \hat\lambda(x)$. Then the functions 
$p_\lambda^-,\,p_\lambda^+:\{a=0\}\to\R$ are well defined, bounded and Lipschitz continuous. More precisely, there exists  a constant $\tilde K_\lambda=\tilde K(\alpha_0,\gamma,\lambda)$, only depending on $\alpha_0>0$, $\gamma>2$ and $\lambda$, and a constant $\tilde \kappa_\lambda=\tilde\kappa(\tilde K_\landa,\alpha_1,\gamma,\eta)$, also depending on $\alpha_1>0$ and $\eta>0$, such that 
\[
|p^\pm_\lambda(x)|\leqslant\tilde K_\lambda
\qquad
\hbox{and}
\qquad
|p^\pm_\lambda(x)-p^\pm_\lambda(y)\leqslant \tilde \kappa_\lambda|x-y|
\qquad\hbox{for all $x,y\in \{a=0\}$.}
\]
\end{prop}

\begin{proof}
Let us denote by $p_\lambda(\cdot)$ either the function $p^-_\lambda(\cdot)$ or the function $p^+_\lambda(\cdot)$. 
The fact that $|p_\lambda(x)|\leqslant\tilde K_\lambda$ for all $x\in \{a=0\}$ is apparent from the fact that $H$ satisfies condition (H1). Let us denote by 
$C=C(\tilde K_\lambda)$ the constant chosen according to (C) above. For every $x,y\in \{a=0\}$ we have 
\[
0
=
|H(x,p_\lambda(x))-H(y,p_\lambda(y))|
\geqslant
|H(y,p_\lambda(x))-H(y,p_\lambda(y))|-C|x-y|
\geqslant 
\eta |p_\lambda(x)-p_\lambda(y)|-C|x-y|,
\]
and the assertion follows by setting $\tilde \kappa_\lambda:=C/\eta$. 
\end{proof}

Next, we establish the uniqueness result mentioned above. 

\begin{theorem}\label{appC teo uniqueness}
Let $\lambda>\lambda_0$, $u_1,u_2$ be two viscosity solutions of
\begin{equation}\label{appC eq2 cellPDE}
	a(x)u''+H(x,u')=\lambda\qquad\hbox{in $\R$,}
	\tag{HJ$_\lambda$}
\end{equation}
and $v_1,v_2$ two viscosity solutions of \eqref{appC eq2 cellPDE} with $\landa:=\landa_0$.
\begin{itemize}
\item[(a)] If $u_i'\geqslant v'_{i}$ a.e. on $\R$ for each $i\in\{1,2\}$, then $u_1-u_2$ is constant on $\R$;\smallskip
\item[(b)] If $u_i'\leqslant v'_{i}$ a.e. on $\R$ for each $i\in\{1,2\}$, then $u_1-u_2$ is constant on $\R$.
\end{itemize}
\end{theorem}

The proof is based on an argument that we state and prove separately.

\begin{prop}\label{appC prop pre-uniqueness}
Let $J$ be a connected component of $A$ and $\lambda>\lambda_0$. Let $f_1,f_2\in\CC^1(J)$ be bounded solutions of the following ODE
\begin{equation}\label{appC eq2 cellODE}
a(x)f'+H(x,f)=\lambda\qquad\hbox{in $J$,}
\end{equation}
and $v_1,v_2$ two viscosity solutions of \eqref{appC eq2 cellPDE} with $\landa:=\landa_0$. 
Let us assume that either $f_i\geqslant v'_{i}$ on $J$ for each $i\in\{1,2\}$,  or $f_i\leqslant v'_{i}$ on $J$ for each $i\in\{1,2\}$. Then $f_1\equiv f_2$. 
\end{prop}

\begin{proof}
By assumption (A1$'$) we know that $J$ is a bounded open interval of the form $J=(\ell_1,\ell_2)$. Let $K>0$ be a constant such that $K\geqslant\|f_i\|_{L^\infty(J)}$ for each $i\in\{1,2\}$. Since the graphs of $f_1$ and $f_2$ cannot cross in $J$, we can assume without any loss of generality that $f_2\geqslant f_1$ in $J$. Let us first consider the case $f_i\geqslant v'_{i}$ on $J$ for each $i\in\{1,2\}$. By Theorem \ref{appC teo boundary regularity2} we know that 
\[
\lim_{x\to\ell_1^+} f_i(x)=p_\lambda^+(\ell_1)>\hat p(\ell_1)\qquad\hbox{for each $i\in \{1,2\}$,}
\]
hence by continuity of the map $\hat p(\cdot)$ we can find $\rho>0$ small enough such that 
\[
f_i(\cdot)>\hat p(\cdot)\quad\hbox{in $I:=(\ell_1,\ell_1+\rho)\subset J$\qquad for each $i\in\{1,2\}$.} 
\]
By subtracting the ODE \eqref{appC eq2 cellODE} for $f_1$ to the one for $f_2$ we get
\[
0=a(x)(f_2-f_1)'(x)+H(x,f_2(x))-H(x,f_1(x)) 
\geqslant
a(x)(f_2-f_1)'(x)+\eta(f_2-f_1)(x)\quad\hbox{in $I$},
\]
where for the last inequality we have used the fact that $f_2(\cdot)\geqslant f_1(\cdot)>\hat p(\cdot)$ in $I$. 
Let us set $w:=f_2-f_1$ and let $y<x$ be a pair of points arbitrarily chosen in $I$. By Gronwall's Lemma we infer 
\[
0\leqslant w(x)
\leqslant 
w(y) e^{-\int_{y}^x \frac{\eta}{a(z)}\, dz}
\leqslant 
2K  e^{-\int_{y}^x \frac{\eta}{a(z)}\, dz}.
\]
By sending $y\to \ell_1^+$ we conclude that $w(x)=0$ for all $x\in I$ since $\int_{y}^x \frac{1}{a(z)}\, dz\to +\infty$ as 
$y\to \ell_1^+$ by Lemma \ref{lemma 1/a}. By the Cauchy-Lipschitz Theorem we conclude that $f_1=f_2$ in $J$. 

The case $f_i\leqslant v'_{i}$ on $J$ for each $i\in\{1,2\}$ can be treated similarly. By Theorem \ref{appC teo boundary regularity2} we know that 
\[
\lim_{x\to\ell_2^-} f_i(x)=p_\lambda^-(\ell_2)<\hat p(\ell_2)\qquad\hbox{for each $i\in \{1,2\}$,}
\]
hence by continuity of the map $\hat p(\cdot)$ we can find $\rho>0$ small enough such that 
\[
f_i(\cdot)<\hat p(\cdot)\quad\hbox{in $J:=(\ell_2-\rho,\ell_2)\subset I$\qquad for each $i\in\{1,2\}$.} 
\]
Now the functions $\check f_i(x):=-f_i(-x)$ solve the ODE  \eqref{appC eq2 cellODE} in $\check J:=(-\ell_2,-\ell_2+\rho)$ with $\check a(x):=a(-x)$ and $\check H(x,p):=H(-x,-p)$ in place of $a$ and $H$, respectively. By the first part of the proof we conclude that $f_1=f_2$ in $J$. 
\end{proof}

\begin{proof}[Proof of Theorem \ref{appC teo uniqueness}.]
Let us prove (a). According to Proposition \ref{prop regularity solutions}, the functions $u_1,u_2$ are Lipschitz on $\R$ and satisfy $\|u_i'\|_\infty\leqslant K$ for each $i\in\{1,2\}$, where $K=K(\lambda,\alpha_0,\alpha_1,\gamma)$ is the constant given explicitly by  \eqref{eq Lipschitz bound}. By Proposition \ref{prop pointwise solution} we have 
\[
H(x,v'_i(x))=\lambda_0<\lambda=H(x,u_i'(x))\qquad\hbox{for a.e. $x\in\{a=0\}$.}
\]
Since $u'_i\geqslant v'_{i}$ a.e. in $\R$, 
by quasiconvexity of $H$ in $p$ we infer that 
\begin{equation}\label{appC eq1 uniqueness}
u_1'(x)=u_2'(x)=p_\lambda^+(x)\qquad\hbox{for a.e. $x\in \{a=0\}$.}
\end{equation}
Let $J$ be a connected component of the set $A=\{a>0\}$.  Each function $u_i$ is of class $C^2$ in $J$ in view of Proposition  \ref{prop regularity solutions}, in particular its derivative $f_i:=u'_i$ is a classical $C^1$ solution of the ODE 
\eqref{appC eq2 cellODE} satisfying $f_i\geqslant v'_i$ in $J$. In view of Proposition \ref{appC prop pre-uniqueness}, we infer that $u_1'=u_2'$ in $J$. By arbitrariness of $J$ and by \eqref{appC eq1 uniqueness}, we conclude that  $u_1'=u_2'$ a.e. in $\R$,   as it was to be shown.

The proof of (b) is similar and it is just sketched. An argument analogous to the one used above shows that 
\begin{equation*} %\label{appC eq2 uniqueness}
u_1'(x)=u_2'(x)=p_\lambda^-(x)\qquad\hbox{for a.e. $x\in \{a=0\}$.}
\end{equation*}
Let $J$ be a connected component of the set $A=\{a>0\}$. Each function $f_i:=u'_i$ is a classical $C^1$ solution of the ODE \eqref{appC eq2 cellODE} satisfying $f_i\leqslant v'_i$ in $J$. In view of Proposition \ref{appC prop pre-uniqueness}, we infer that $u_1'=u_2'$ in $J$. This eventually shows that $u_1'=u_2'$ a.e. in $\R$, as it was asserted.  
\end{proof}

We collect all the information gathered so far in the following key result. 

\begin{theorem}\label{appC teo existence corrector}
Let $\lambda>\lambda_0$ and let $u_{\landa_0}$ be a Lipchitz viscosity solution of \eqref{appC eq cellPDE} with $	\lambda=\landa_0$. 
\begin{itemize}
\item[(a)] There exists a unique viscosity solution $u^+_\lambda$ to \eqref{appC eq2 cellPDE} satisfying 
$(u^+_\lambda)' > u'_{\lambda_0}$ a.e. on $\R$ and $u^+_\lambda(0)=0$. The function $u^+_\lambda$ is of class $C^1$ on $\R$, of class $C^2$ on $A:=\{a>0\}$ and satisfies $(u^+_\lambda)'=p^+_\landa(\cdot)$ on $\{a=0\}$, in particular it is a classical solution of \eqref{appC eq2 cellPDE}. Furthermore, $u^+_\lambda$ does not depend on the choice of $u_{\landa_0}$.
\smallskip
\item[(b)] There exists a unique viscosity solution $u^-_\lambda$ to \eqref{appC eq2 cellPDE} satisfying 
$(u^-_\lambda)'< u'_{\lambda_0}$ a.e. on $\R$ and $u^-_\lambda(0)=0$. The function $u^-_\lambda$ is of class $C^1$ on $\R$, of class $C^2$ on $A:=\{a>0\}$ and satisfies $(u^-_\lambda)'=p^-_\landa(\cdot)$ on $\{a=0\}$, in particular it is a classical solution of \eqref{appC eq2 cellPDE}. Furthermore, $u^-_\lambda$ does not depend on the choice of $u_{\landa_0}$.
\end{itemize}
\end{theorem}

\begin{proof}
We only need to prove the existence part, since uniqueness and independence from $u_{\landa_0}$ is guaranteed by Theorem \ref{appC teo uniqueness}. 
Let us prove item (a). According to Proposition \ref{appC prop pre-existence}, there exists a bounded function $f^+_\lambda\in C^1(A)$ such that $f^+_\lambda > u'_{\lambda_0}$ in $A$ and $f^+_\lambda$ solves 
\begin{equation*}
a(x)f'+H(x,f)=\lambda\qquad\hbox{in $A=\{a>0\}$.}
\end{equation*}
We extend $f^+_\lambda$  to the whole $\R$ by setting $f^+_\lambda(x):=p^+_\lambda(x)$ for all $x\in\{a=0\}$. This is well defined in view of Propositions \ref{appC prop lambda 0} and \ref{appC prop p lambda}. We claim that the function $f^+_\lambda$ is continuous on $\R$. Indeed, let $x_0\in \R$ and $(x_n)_n$ be a sequence converging to $x_0$. Since $f^+_\lambda$ is $C^1$ in $A$ and $A$ is open, we can assume that $x_0\in\{a=0\}$. In view of Propositions \eqref{appC prop lambda 0} and  \ref{appC prop p lambda} and up to extracting a subsequence, we can furthermore assume that $x_n\in A$ for every $n\in\N$. 
This means that $x_0$ is an accumulation point for $A$. From Theorem \ref{appC teo boundary regularity2} we infer that $\lim_n f^+_\lambda(x_n)= p^+_\lambda(x_0)=f^+_\lambda(x_0)$. This shows the asserted continuity of $f^+_\lambda$ in $\R$. The sought solution is defined by setting $u^+_\lambda(x):=\int_0^x f^+_\lambda(z)\, dz$ for all $x\in\R$. 
The proof of item (b) is similar as is omitted. 
\end{proof}

We proceed by showing the following properties of the functions $\landa\mapsto (u_\landa)'$.

\begin{prop}\label{appC prop monotonicity derivatives}
The following properties hold.
\begin{itemize}
\item[\em(i)] Let $\lambda>\mu>\lambda_0$. Then \ 
$(u^+_\lambda)'>(u^+_\mu)'>(u^-_\mu)'>(u^-_\lambda)'$\  \ in $\R$.\smallskip
\item[\em(ii)] Let $\lambda>\lambda_0$. Then 
\[
\lim_{\mu\to\lambda} (u^+_\mu)'(x)=(u^+_\lambda)'(x),
\qquad
\lim_{\mu\to\lambda} (u^-_\mu)'(x)=(u^-_\lambda)'(x)
\qquad
\hbox{for every $x\in\R$.}
\]
\item[\em(iii)] For every $R>\|u'_{\landa_0}\|_{L^\infty(A)}$ there exists $\landa_R>\landa_0$ such that for every $\landa>\landa_R$ we have 
\[
(u^-_{\lambda})'(x)<-R<R<(u^+_{\lambda})'(x)
\qquad
\hbox{for all $x\in\R$.}
\]
\end{itemize}
\end{prop}

\begin{proof} (i) Let us set $f_\lambda:=(u^+_\lambda)'$ for every $\lambda>\lambda_0$. By Theorem \ref{appC teo existence corrector} we know that $f_\lambda\in\CC(\R)$ and $f_\lambda=p^+_\lambda$ on $\{a=0\}=\R\setminus A$. Hence $f_\lambda>f_\mu$ on $\R\setminus A$ if $\lambda>\mu>\lambda_0$. Let $J=(\ell_1,\ell_2)$ be a connected component of the set $A$.  Since $f_\lambda$ is a $C^1$ solution of \eqref{appC eq2 cellODE}, at any point $x_0\in J$ such that $f_\lambda(x_0)=f_\mu(x_0)$ we would have $(f_\lambda-f_\mu)'(x_0)=(\lambda-\mu)/a(x_0)>0$. Since $(f_\lambda-f_\mu)(\ell_1)=p^+_\lambda(\ell_1)-p^+_\mu(\ell_1)>0$, this implies that $f_\lambda>f_\mu$ in $J$, as it was asserted. The statement for $(u^-_\mu)'$ and $(u^-_\lambda)'$ can be proved analogously. 

(ii) Fix $\lambda>\lambda_0$ and let $(\lambda_n)_n$ be an increasing sequence in $(\lambda_0,+\infty)$ converging to $\lambda$ from the left. By item (i), we know that for every fixed $y\in\R$, the sequence $\big(u^+_{\lambda_n}(x)-u^+_{\lambda_n}(y)\big)_n$ is decreasing if $x<y$ and increasing if $x>y$, in particular 
$\lim_n u^+_{\lambda_n}(x)-u^+_{\lambda_n}(y)=:u(x)$ exists for every $x\in\R$. Since the function $(u^+_{\lambda_n})_n$ are equi-Lipschitz in $\R$ in view of Proposition \ref{prop regularity solutions}, this convergence is also locally uniform in $\R$. By the stability of the notion of viscosity solution, we infer that $u$ is a solution of \eqref{appC eq2 cellPDE}. Since clearly $u'\geqslant u'_{\lambda_0}$ a.e. in $\R$, by Theorem \ref{appC teo uniqueness} we conclude that 
$u(x)=u^+_\lambda(x)-u^+_\lambda(y)$ for every $x\in\R$. Let us set $g(x):=\sup_n (u^+_{\lambda_n})'(x)=\lim_n (u^+_{\lambda_n})'(x)$ for every $x\in\R$. By item (i), we know that $g\leqslant (u^+_\lambda)'$ in $\R$. On the other hand, for every $x,y\in\R$ we have
\begin{eqnarray*}
\int_y^ x (u^+_\lambda)'(z)\, dz
&=&
u^+_\lambda(x)-u^+_\lambda(y)
=
\lim_{n\to +\infty} u^+_{\lambda_n}(x)-u^+_{\lambda_n}(y)\\
&=&
\lim_{n\to +\infty} \int_y^x (u^+_{\lambda_n})'(z)\, dz
=
\int_y^x \lim_{n\to +\infty} (u^+_{\lambda_n})'(z)\, dz
=
\int_y^x g(z)\, dz,
\end{eqnarray*}
where for the second to last equality we have used the Dominated Convergence Theorem. This yields $(u^+_\lambda)'=g=\lim_n (u^+_{\lambda_n})'$ in $\R$, as it is easily seen, and finally shows that $\lim_{\mu\to\lambda^-} (u^+_\mu)'=(u^+_\lambda)'$ in $\R$. The proof of the remaining cases is analogous and is omitted.

(iii) The argument is analogous to the one used in the proof of Proposition \ref{appC prop pre-existence}. Let us set $\landa_R:=\sup_{x\in\R} H(x,R)$. Pick $\landa>\landa_R$ and choose $q_\landa>R$ such that $\inf_{x\in\R} H(x,q_\landa)>\landa$. Then the constant functions $m(\cdot):=R$ and $M(\cdot):=q_\landa$ satisfy \eqref{eq appC pre-existence} and $m(\cdot)<M(\cdot)$ in $J$, for any connected component $J$ of $A$.  By arguing as in the proof of Proposition \ref{appC prop pre-existence}, we find a solution $f_\landa\in\CC^1(A)$ of \eqref{appC eq2 cellODE} in $A$ satisfying $m(\cdot)<f_\landa<M(\cdot)$ in $A$. We extend 
$f_\lambda$  to the whole $\R$ by setting $f_\lambda(x):=p^+_\lambda(x)$ for all $x\in\{a=0\}$. By arguing as in the proof of Theorem \ref{appC teo existence corrector}, we derive that $f_\landa$ is continuous in $\R$, hence $u_\landa(x):=\int_0^x f_\landa(z)\,dz$, $x\in\R$, is a $C^1$ solution of \eqref{appC eq2 cellPDE} satisfying $(u_\landa)'>R>u'_{\landa_0}$ almost everywhere in $\R$. By uniqueness, see Theorem \ref{appC teo existence corrector}, we get $u^+_\landa\equiv u_\landa$ and we conclude that 
$(u^+_{\lambda})'(x)=f_\landa(x)>R$ for all $x\in\R$, as it was asserted. The proof for $(u^-_{\lambda})'$ is similar and is omitted.
\end{proof}

We conclude this subsection by proving a fact that will play a crucial role in the proof of Theorem \ref{teo lower bound}. 

\begin{theorem}\label{appC teo supersolution}
Let $\lambda>\sup_{x\in\{a=0\}} \hat\lambda(x)$ and assume that there exists a bounded function $f\in C^1(A)$ satisfying 
\begin{equation}\label{appC eq3 cellODE}
a(x)f'+H(x,f)=\lambda\qquad\hbox{in $A=\{a>0\}$.}
\end{equation}
Let us extend $f$ to the whole $\R$ by setting $f(x)=p_\lambda^+(x)$ for all $x\in \{a=0\}$. Then the function 
$u(x):=\int_0^x f(z)\,dz$, $x\in\R$, is a Lipschitz viscosity supersolution of \eqref{appC eq2 cellPDE}.
\end{theorem}

\begin{proof}
By assumption and Proposition \ref{appC prop p lambda}, there exists a constant 
$K$ such that $|f(x)|\leqslant K$ for all $x\in\R$. Hence the function $u$ is globally Lipschitz on $\R$. The function $u$ is of class $C^2$ and a classical solution to \eqref{appC eq2 cellPDE} in $A$. 
It is left to show that, for any $x\in\{a=0\}$, we have 
\begin{equation}\label{appC eq1 supersolution}
H(x,p)\leqslant \lambda\qquad\hbox{for all $p\in D^-u(x)$.}
\end{equation}
Since $D^- u(x)$ is contained in the Clarke's generalized gradient $\partial^c u(x)$, see \cite[Proposition 2.3.2]{Clarke},  showing \eqref{appC eq1 supersolution} for every $p\in \partial^c u(x)$ implies the assertion. We recall that the Clarke's generalized gradient 
$\partial^c u(x)$ can be defined, for every $x\in\R$,  as the convex hull of the closed set $D_\Sigma^*u(x)$ of the following set of reachable gradients
\[
D_\Sigma^*u(x):=\left\{p\in\R\,:\, p=\lim_n u'(x_n)\ \hbox{for some sequence $(x_n)_n$ in $\R\setminus\Sigma$ converging to $x$}\, \right\},
\]
where $\Sigma$ is any fixed negligible subset of $\R$ containing the nondifferentiability points of $u$. It is known that the convex hull of $D^*_\Sigma(x)$ is independent of the choice of $\Sigma$, see \cite[Theorem 2.5.1]{Clarke}. Since the function $H(x,\cdot)$ is quasiconvex, it will be therefore enough to prove \eqref{appC eq1 supersolution} for any $p\in D_\Sigma^*u(x)$ for a 
proper choice of $\Sigma$. 
Let us choose $\Sigma$ in such a way that any $x\in\R\setminus\Sigma$ is a Lebesgue point for $f$. Clearly, for every $x\in\R\setminus\Sigma$,  
\[
\lim_{y\to x}\left|\dfrac{u(y)-u(x)}{y-x}-f(x)\right|
=
\lim_{y\to x}\dfrac{1}{|y-x|}\left|\int_x^y (f(z)-f(x))\,dz\right|=0,
\]
namely $u$ is differentiable at $x$ and $u'(x)=f(x)$. Let us fix $x_0\in\{a=0\}$. Pick $p_0\in D^*_\Sigma u(x_0)$ and 
let $(x_n)_n$ be a sequence in $\R\setminus\Sigma$ converging to $x_0$ and such that $\lim_n u'(x_n)=p_0$. 
Up to extracting a subsequence, we can assume that  $(x_n)_n$ is either contained in the set $\{a=0\}$  or in its complement $A$. In the first case, we have $u'(x_n)=p^+_\lambda(x_n)$ for all $n\in\N$, so $p_0=p^+_\lambda(x_0)$ in view of Proposition \ref{appC prop p lambda}, and the point $(x_0,p_0)$ satisfies the equality in \eqref{appC eq1 supersolution}. Let us then assume that $(x_n)_n$ is contained in $A$. If there exists $n_0\in\N$ such that, for every $n\geqslant n_0$, the points $x_n$ are contained in the same connected component $J$ of $A$, then $x_0$ is an end point of the interval $J$. By Theorem \ref{appC teo boundary regularity} we conclude that $H(x_0,p_0)=\lim_n H(x_n,f(x_n))=\lambda$. Otherwise, up to extracting a subsequence if necessary, we can assume that each $x_n$ is contained in an interval $J_n$, where $(J_n)_n$ is a sequence of pairwise disjoint connected components of $A$. 
Let $C=C(K)$ be the constant chosen according to (C) above and set $\delta:=\lambda-\hat\lambda(x_0)>0$. Denote by $r(\hat\kappa,C,\delta)>0$ the radius given by Lemma \ref{appC lemma 2} and pick  $r\in\big(0,r(\hat\kappa,C,\delta)\big)$. According to Proposition \ref{appC prop tool2+}, we have $H(x_n,f(x_n))\leqslant \lambda+2rC$ for every $x_n\in (x_0-r,x_0+r)$. By sending $n\to +\infty$ and by arbitrariness of $r>0$, we conclude that $H(x_0,p_0)\leqslant \lambda$. 
This concludes the proof.
\end{proof}

\section{The homogenization result}\label{sec:homogenization}

In this section we will prove the homogenization result stated in Theorem \ref{thm:genhom2}. We will therefore assume that, for every $\omega\in\Omega$, the Hamiltonian $H(\cdot,\cdot,\omega)$ belong to $\Hamsqc$ for some fixed constants $\alpha_0,\alpha_1>0$, $\gamma>2$ and $\eta>0$. 

\subsection{Random correctors} In this subsection we will construct suitable random correctors of the following viscous HJ equation, for every $\lambda>\lambda_0$:
\begin{equation}\label{eq2 cellPDE}
	a(x,\omega)u''+H(x,u',\omega)=\lambda\qquad\hbox{in $\R$.}
	\tag{HJ$_\lambda(\omega)$}
\end{equation}
In view of Lemma \ref{lemma lambda zero} and Lemma \ref{lemma A(omega)}, there exists a set $\hat\Omega$ of probability 1 such that, for every $\omega\in\hat\Omega$, the function $a(\cdot,\omega)$ satisfies assumption (A1$'$) and $\landa_0(\omega)=\landa_0$.  For every fixed $\omega\in\hat\Omega$, we denote by $u_{\lambda_0}(\cdot,\omega)$ a Lipschitz solution of equation \eqref{eq2 cellPDE} with $\landa=\landa_0$ and satisfying $u_{\landa_0}(0,\omega)=0$. We extend the definition of $u_{\landa_0}$ to $\Omega\times\R$ by setting $u_{\landa_0}(\cdot,\landa)\equiv 0$ for $\omega\in\Omega\setminus\hat\Omega$. We  point out that, in the current definition of $u_{\landa_0}(\cdot,\omega)$, we are treating the variable $\omega$ as a fixed parameter, in particular the function $u_{\landa_0}:\R\times\Omega\to\R$ is not jointly-measurable, in general. 

\smallskip

For $\lambda>\lambda_0$ and $\omega\in\Omega$ fixed, we denote by
\begin{eqnarray*}
\Sol^-_\lambda(\omega) &:=& \{u\in\Lip(\R)\,:\, \text{$u$ solves \eqref{eq2 cellPDE} and 
$u'\leqslant u'_{\landa_0}$ a.e. in $\R$}\},\\
\Sol^+_\lambda(\omega) &:=& \{u\in\Lip(\R)\,:\, \text{$u$ solves \eqref{eq2 cellPDE} and 
$u'\geqslant u'_{\landa_0}$ a.e. in $\R$}\},\bigskip
\end{eqnarray*}
where we agree that $\Sol^\pm_\lambda(\omega)$ is the set of constant functions on $\R$ when $\omega\in\Omega\setminus\hat\Omega$. The following holds.

\begin{theorem}\label{teo1 random correctors}
Let $\lambda>\lambda_0$  be fixed. The following holds:
\begin{itemize}
\item[(a)] for every $\omega\in\Omega$, there exists a unique $u^+_\lambda\in\Sol^+_\lambda(\omega)$ with $u^+_\lambda(0,\omega)=0$. Furthermore, $u^+_\lambda(\cdot,\omega)\in\CC^1(\R)$;\smallskip
\item[(b)] for every $\omega\in\Omega$, there exists a unique $u^-_\lambda\in\Sol^-_\lambda(\omega)$ with $u^-_\lambda(0,\omega)=0$. Furthermore, $u^-_\lambda(\cdot,\omega)\in\CC^1(\R)$.\medskip
\end{itemize}
Moreover, the functions $u^\lambda_+,u^\lambda_-:\R\times\Omega\to\R$ are jointly measurable and have stationary derivatives. 
\end{theorem}

\begin{proof}
Assertion (a) and (b) are obvious if $\omega\in\Omega\setminus\hat\Omega$. If $\omega\in\hat\Omega$, they follow directly from Theorem \ref{appC teo existence corrector}.  
Let us prove that $u^+_\lambda:\R\times\Omega\to\R$ is measurable with respect to the product $\sigma$-algebra ${\mathcal B}\otimes\F$. This is equivalent to showing that $\Omega\ni \omega\mapsto u^+_\lambda(\,\cdot\,,\omega)\in\CC(\R)$ is a random variable from $(\Omega,\F)$ to the Polish space $\CC(\R)$ endowed with its Borel $\sigma$-algebra, see for instance  \cite[Proposition 2.1]{DS09}. 
Since the probability measure $\P$ is complete on $(\Omega,\F)$, it is enough to show that, for every fixed $\eps>0$, there exists a set $F\in \F$ with $\P(\Omega\setminus F)<\eps$ such that the restriction 
$u^+_\lambda$ to $F$ is a random variable from $F$ to $\CC(\R)$.  To this aim, we notice that the measure $\P$ is inner regular on $(\Omega,\F)$, see \cite[Theorem 1.3]{Bill99}, hence it is a Radon measure. By applying Lusin's Theorem \cite{LusinThm} to the random variable $a:\Omega\to\CC(\R)$ and $H:\Omega\to\CC(\R\times\R)$, we infer that there exists a closed set $F\subseteq \Omega$ with $\P(\Omega\setminus F)<\eps$ such that $a_{| F}:F\to\CC(\R)$ and $H_{| F}:F\to\CC(\R\times\R)$ are continuous. 
We claim that $F\ni\omega\mapsto u^+_\lambda(\,\cdot\,,\omega)\in\CC(\R)$ is continuous. 
Indeed, let $(\omega_n)_{n\in\N}$ be a sequence converging to some $\omega_0$ in $F$. The functions $u^+_\lambda(\cdot,\omega_n)$ are equi-Lipschitz and locally equi-bounded in $\R$, hence they converge, up to subsequences, to a Lipschitz function $u$ satisfying $u(0)=0$. Since $a(\cdot,\omega_n)\to a(\cdot,\omega_0)$ in 
$\CC(\R)$ and 
$H(\cdot,\cdot,\omega_n)\to H(\cdot,\cdot,\omega_0)$ in $\CC(\R\times\R)$, we derive by stability that $u$ solves \eqref{eq2 cellPDE} with $\omega:=\omega_0$ in the viscosity sense. Via the same argument, we see that the functions $u_{\lambda_0}(\cdot,\omega_n)$ converge, up to subsequences, to a solution $v$ of \eqref{eq2 cellPDE} with $\omega:=\omega_0$ and $\landa=\landa_0$. The inequality $(u^+_\landa)'(\cdot,\omega_n)\geqslant  u'_{\landa_0}(\cdot,\omega_n)$ a.e. in $\R$  amounts to requiring that 
$\left(u^+_\landa-u_{\landa_0}\right)(\cdot,\omega_n)$ is nondecreasing in $\R$. Since this property is stable under pointwise convergence, we derive that $u'\geqslant v'$ a.e. in $\R$. By uniqueness, see Theorems \ref{appC teo uniqueness} and \ref{appC teo existence corrector},  we infer that $u\equiv u^+_\lambda(\cdot,\omega_0)$, yielding in particular that the whole sequence $\big(u^+_\lambda(\cdot,\omega_n)\big)_n$ is converging to $u^+_\lambda(\cdot,\omega_0)$. This proves the asserted continuity property of the map $\omega\mapsto u^+_\lambda(\cdot,\omega)$ on $F$. 

Last, for every fixed $z\in\R$ and $\omega\in\Omega$, the map $u(\cdot):=u^+_\lambda(\cdot+z,\omega)-u^+_\lambda(z,\omega)$ is a solution of equation \eqref{eq2 cellPDE} with $\tau_z\omega$ in place of $\omega$, by the stationary character of $a$ and $H$. Furthermore, it satisfies $u(0)=0$ and $u'\geqslant u'_{\landa_0}(\cdot+z,\omega)$ a.e. in $\R$. By the stationary character of $a$ and $H$, the function $v(\cdot):=u_{\landa_0}(\cdot+z,\omega)$ is a solution  of equation \eqref{eq2 cellPDE} with $\tau_z\omega$ in place of $\omega$ and $\landa_0$ in place of $\landa$. By uniqueness, we get $u=u^+_\lambda(\cdot,\tau_z\omega)$, i.e., 
$u^+_\lambda(\cdot+z,\omega)-u^+_\lambda(z,\omega)=u^+_\lambda(\cdot,\tau_z\omega)$. This clearly implies that $u^+_\lambda$ has stationary derivative. 
\end{proof}

From the information gathered in Section \ref{sec:deterministic correctors main}, we easily derive the following facts. 

\begin{prop}\label{prop theta functions}\ 
\begin{itemize}
\item[\em(i)] The quantity $\theta^+(\lambda):=\EE \big[ (u^+_\lambda)'(0,\omega) \big]$  defines a continuous, coercive and strictly increasing function
\[
\theta^+ :[\lambda_0,+\infty)\to \big(\theta^+(\lambda_0),+\infty\big),
\]
where $\theta^+(\lambda_0):=\inf_{\landa>\landa_0} \theta^+(\lambda)$. 
\item[\em(ii)] The quantity $\theta^-(\lambda):=\EE \big[ (u^-_\lambda)'(0,\omega) \big]$  defines a continuous, coercive and strictly decreasing function
\[
\theta^- :[\lambda_0,+\infty)\to \big(-\infty,\theta^-(\lambda_0)\big),
\]
where $\theta^-(\lambda_0):=\sup_{\landa>\landa_0} \theta^-(\lambda)$. 
\end{itemize}
\end{prop}

\begin{proof}
The assertion follows as a direct application of   Proposition \ref{appC prop monotonicity derivatives} and of the Dominated Convergence Theorem in view of Proposition \ref{prop regularity solutions}. 
\end{proof}

\subsection{The flat part}\label{sec:flat part}
This subsection is addressed at proving homogenization at the bottom with the existence of a possible flat part. More precisely, we want to prove that 
\[
\HV^L(H)(\theta)=\HV^U(H)(\theta)=\lambda_0
\qquad 
\hbox{for all\quad $\theta \in \big[\theta^-(\lambda_0),\theta^+(\lambda_0)\big]$.}
\]
We distill in the next statement an argument that already appeared in \cite{DK22} and which will be needed 
for the proofs of Theorems \ref{teo1 upper qc} and \ref{teo lower bound}.

\begin{prop}\label{prop corrector argument}
Let $u_1,u_2$ be random viscosity solutions of \eqref{eq2 cellPDE} for the same $\lambda$ with stationary gradient and satisfying $u_1'(\cdot,\omega)\leqslant u'_2(\cdot,\omega)$ a.e. in $\R$,  almost surely. 
Let us set $\theta_i:=\EE[u'_i(0,\omega)]$ for $i\in\{1,2\}$. 
\begin{itemize}
\item[\em(i)] Let us assume that there exist $\mu>\lambda$ and a set $\Omega_\mu$ of positive probability such that, for every $\omega\in\Omega_\mu$, there exists a locally Lipschitz function $w$ satisfying  
\[
a(x,\omega)u''+H(x,u',\omega)\leqslant \mu\qquad\hbox{in $\R$}
\]
and such that $w'=u'_1(\cdot,\omega)$ a.e. in $(-\infty, -L)$ and $w'=u'_2(\cdot,\omega)$ a.e. in $(L,+\infty)$ for some $L>0$. Then \ $\HV^U(H)(\theta)\leqslant \mu$\quad for every $\theta\in (\theta_1,\theta_2)$.\smallskip
\item[\em(ii)] Let us assume that there exist $\mu<\lambda$ and a set $\Omega_\mu$  of positive probability such that, for every $\omega\in\Omega_\mu$, there exists a locally Lipschitz function $v$ satisfying  
\[
a(x,\omega)u''+H(x,u',\omega)\geqslant \mu\qquad\hbox{in $\R$}
\]
and such that $v'=u'_2(\cdot,\omega)$ a.e. in $(-\infty, -L)$ and $v'=u'_1(\cdot,\omega)$ a.e. in $(L,+\infty)$ for some $L>0$. Then \ $\HV^L(H)(\theta)\geqslant \mu$\quad for every $\theta\in (\theta_1,\theta_2)$.
\end{itemize}
\end{prop}

\begin{proof}
The proof is based on an argument already used in \cite[Lemma 5.6]{DK22}. We shall only prove (i), being the proof of (ii) analogous. Let us fix $\theta\in (\theta_1,\theta_2)$ and choose a set $\hat\Omega$ of probability 1 such that, for every $\omega\in\hat\Omega$, both the limits in \eqref{eq:infsup}  and the following ones (by Birkhoff Ergodic Theorem) hold:
\begin{align}
\lim_{x\to+\infty}\frac{1}{x}\int_{L}^x{u_2'}(s,\omega)\,ds&=\EE[u_2'(0,\omega)] =\theta_2> \theta\label{eq1 Birkhoff}\\
\lim_{x\to-\infty}\frac{1}{|x|}\int_{x}^{-L}{u_1'}(s,\omega)\,ds&=\EE[u_1'(0,\omega)]=\theta_1 < \theta.\label{eq2 Birkhoff}
\end{align}
Let us fix  $\omega\in\hat\Omega\cap\Omega_\mu$ and set 
$
\  \tilde w(t,x)=\mu t+w(x)+M\ \  \hbox{for all $(t,x)\in \ccyl$},
$
where the constant $M$ is chosen large enough to ensure that $\tilde w(0,x)\geqslant \theta x$ for all $x\in\R$. This is possible due to \eqref{eq1 Birkhoff} and \eqref{eq2 Birkhoff} and since, by assumption,  $w'=u'_1(\cdot,\omega)$ a.e. in $(-\infty, -L)$ and a.e. in $w'=u'_2(\cdot,\omega)$ in $(L,+\infty)$ for some $L>0$.
The function $\tilde w$ is a supersolution of \eqref{eq:generalHJ}. Indeed, 
	\begin{align*}
		\partial_t\tilde w
=
		\mu \ge	
a(x,\omega)\partial_{xx}^2\tilde w + H(x,\partial_x \tilde w,\omega)
=
a(x,\omega)w'' + H(x, w',\omega)
\quad
\hbox{in $\ccyl$}.
\end{align*}
By the comparison principle, $\tilde w(t,x)\geqslant u_\theta(t,x,\omega)$ on $\ccyl$ and, hence,
	\[\HV^U(H)(\theta)=\limsup_{t\to+\infty}\frac{u_\theta(t,0,\omega)}{t}\leqslant\limsup_{t\to+\infty}\frac{\tilde w(t,0)}{t}=\mu.\qedhere
\] 
\end{proof}

We proceed by showing the first of the two main results of this subsection. It expresses the fact that the function $\HV^U(H)$ enjoys a quasiconvexity-type property. 

\begin{theorem}\label{teo1 upper qc}
Let $\lambda>\lambda_0$. Then\quad $\HV^U(H)(\theta)\leqslant \lambda\quad\hbox{for all $\theta\in [\theta^-(\lambda),\theta^+(\lambda)]$.}$
\end{theorem}

\begin{proof}
Let $\hat\Omega$ be a set of probability 1 such that 
$\lambda_0(\omega)=\lambda_0$, the diffusion coefficient 
$a(\cdot,\omega)$ satisfies condition (A1$'$), and the equalities in \eqref{eq:infsup} holds, 
for every $\omega\in\hat\Omega$. 

Let us fix $\omega\in\hat\Omega$ and pick $x_0\in\{a(\cdot,\omega)=0\}$. 
Let us set $f^-(\cdot):=(u^-_\lambda)'(\cdot,\omega)$ and $f^+(\cdot):=(u^+_\lambda)'(\cdot,\omega)$ in $\R$. 
In view of Proposition \ref{prop pointwise solution} we know that  
\[
H\big(x_0,f^\pm(x_0),\omega\big)
=
H\big(x_0,(u^\pm_\lambda)'(x_0,\omega),\omega\big)=\lambda.
\]
Define a Lipschitz function $w$ on $\R$ by setting 
\[
w(x):=\int_{x_0}^x \Big( f^-(z,\omega)\1_{(-\infty,x_0]}(z)+f^+(z,\omega)\1_{(x_0,+\infty)}(z) \Big)\,dz,
\qquad
x\in\R.
\footnote{We denote by $\1_E$ the characteristic function of the set $E$, i.e., the function which is identically equal to 1 on $E$ and to 0 in its complement.}
\]
Then $w=u^-_\lambda(\cdot,\omega)-u^-_\lambda(x_0,\omega)$ in $(-\infty,x_0)$ and 
$w=u^+_\lambda(\cdot,\omega)-u^+_\lambda(x_0,\omega)$  in $(x_0,+\infty)$, hence $w$ solves \eqref{eq2 cellPDE} in $\R\setminus\{x_0\}$. On the other hand, the subdifferential of $w$ the point $x_0$ is  
$D^-w(x_0)=[f^-(x_0),f^+(x_0) ]$. Since any $C^2$--subtangent $\varphi$ to $w$ at $x_0$ satisfies $\varphi'(x_0)\in D^-w(x_0)$, by the quasiconvex character of $H(x_0,\cdot,\omega)$ we infer  
\begin{equation*}
a(x_0,\omega)\varphi''(x_0)+H\big(x_0,\varphi'(x_0),\omega\big)
\leqslant 
\max\left\{H\big(x_0,f^-(x_0),\omega\big), H\big(x_0,f^+(x_0),\omega\big)\right\}=
\lambda.
\end{equation*}
We conclude that $w$ is a viscosity supersolution of \eqref{eq2 cellPDE} satisfying $w'=(u^-_\lambda)'(\cdot,\omega)$ in $(-\infty,x_0)$ and $w'=(u^+_\lambda)'(\cdot,\omega)$ in $(x_0,+\infty)$. The assertion follows in view of Proposition \ref{prop corrector argument}-(i).
\end{proof}

We now prove the second main result of this subsection. 
For this, we exploit the definition of  $\landa_0$ to provide the expected lower bound  the function $\HV^L(H)$. The core of the proof consists in exploiting the absence of (deterministic) supersolutions of \eqref{eq2 cellPDE} for values of the constant at the right-hand side lower than $\lambda_0$ in order to construct a lift which allows to gently descend from $(u^+_\lambda)'(\cdot,\omega)$ to $(u^-_\lambda)'(\cdot,\omega)$, almost surely and for every fixed $\landa>\landa_0$.\smallskip

\begin{theorem}\label{teo lower bound}
We have \ $\HV^L(H)(\theta)\geqslant \lambda_0$ \ for all $\theta\in [\theta^-(\lambda_0),\theta^+(\lambda_0)]$.
\end{theorem}

\begin{proof}
By stationarity of the $a$ and $H$ and the ergodicity assumption, we have two possible different scenarios.\smallskip

\noindent{{\bf Case 1:}} {\em there exists a set $\hat \Omega$ of probability 1 such that for every $\omega\in\hat\Omega$ the following property holds:
\begin{itemize}
\item\ for every $n\in\N$, there exists a point $x_n\in\{a(\cdot,\omega)=0\}$ such that 
\begin{equation*}
|\{p\in\R\,:\, H(x_n,p,\omega)\leqslant \lambda_0\}|<\frac1n.
\end{equation*}
\end{itemize}
}
\noindent
Up to choosing a smaller $\hat\Omega$ if necessary, we can assume that 
$\lambda_0(\omega)=\lambda_0$ and $a(\cdot,\omega)$ satisfies (A1$'$) for every $\omega\in\hat\Omega$. 
Let us fix $\omega\in\hat\Omega$ and $n\in\N$. Choose $\lambda>\lambda_0$ close enough to $\lambda_0$ such that 
\begin{equation}\label{eq1.1 lower bound}
|\{p\in\R\,:\, H(x_n,p,\omega)\leqslant \lambda\}|<\frac1n.
\end{equation}
Then the functions $u^\pm_\lambda(\cdot,\omega)$ are  viscosity solutions of \eqref{eq2 cellPDE} of class $C^1$ on $\R$.
In particular 
\[
%H(x_n,f^\pm(x_n),\omega)=
H(x_n,(u^\pm_\lambda)'(x_n,\omega),\omega)=\lambda.
\]
%where we have set $f^-:=(u^-_\lambda)'(\cdot,\omega)$ and $f^+:=(u^+_\lambda)'(\cdot,\omega)$. 
Define a Lipschitz function $v$ on $\R$ by setting 
\[
v(x):=\int_{x_n}^x \Big( (u^+_\lambda)'(z,\omega)\1_{(-\infty,x_n]}(z)+(u^-_\lambda)'(z,\omega)\1_{(x_n,+\infty)}(z) \Big)\,dz,
\qquad
x\in\R.
\]
Then $w=u^+_\lambda(\cdot,\omega)-u^+_\lambda(x_n,\omega)$ in $(-\infty,x_n)$ and 
$w=u^-_\lambda(\cdot,\omega)-u^-_\lambda(x_n,\omega)$  in $(x_n,+\infty)$, hence $v$ solves \eqref{eq2 cellPDE} in $\R\setminus\{x_n\}$. On the other hand, the superdifferential of $v$ the point $x_n$ is  
\[
D^+w(x_n)=[(u^-_\lambda)'(x_n,\omega),(u^+_\lambda)'(x_n,\omega)].
%=
%\{p\in\R\,:\, H(x_n,p,\omega)\leqslant \lambda\}.
\]
Let us denote by $\tilde C$ a Lipschitz constant of $H(x_n,\cdot,\omega)$ in an open neighborhood of 
$\{p\in\R\,:\, H(x_n,p,\omega)\leqslant \lambda\}$. Since any $C^2$--supertangent $\varphi$ to $v$ at $x_n$ satisfies $\varphi'(x_n)\in D^+w(x_n)$, we derive from \eqref{eq1.1 lower bound}
\begin{equation*}
a(x_n,\omega)\varphi''(x_n)+H\big(x_n,\varphi'(x_n),\omega\big)
=
H\big(x_n,\varphi'(x_n),\omega\big)
\geqslant
H\big(x_n,(u^+_\lambda)'(x_n,\omega),\omega\big)-\frac{\tilde C}{n}
=
\lambda -\frac{\tilde C}{n}.
\end{equation*}
Hence $v$ is a subsolution of \eqref{eq2 cellPDE} with $\lambda-\tilde C/n$ in place of $\lambda$ satisfying 
$w'=(u^+_\lambda)'(\cdot,\omega)$ in $(-\infty,x_n)$ and $w'=(u^-_\lambda)'(\cdot,\omega)$ in $(x_n,+\infty)$. 
Pick $\theta\in [\theta^-(\lambda_0),\theta^+(\lambda_0)]\subsetneq [\theta^-(\lambda),\theta^+(\lambda)]$.
By Proposition \ref{prop corrector argument}-(ii) we get \ $\HV(H)(\theta)\geqslant \lambda-\tilde C/n$. By letting 
$\lambda\searrow \lambda_0$ and then $n\to +\infty$ we get the assertion.\medskip

\noindent{\bf Case 2:}{\em \ there exists a set $\hat \Omega$ of probability 1 and $\eps_0>0$ such that, 
for every $\omega\in\hat\Omega$, we have}
\[
|\{p\in\R\,:\, H(x,p,\omega)\leqslant \lambda_0\}|\geqslant \eps_0\qquad\hbox{for all $x\in\{a(\cdot,\omega)=0\}$.}
\] 
\noindent
Up to choosing a smaller $\hat\Omega$ if necessary, we can assume that 
$\lambda_0(\omega)=\lambda_0$ and $a(\cdot,\omega)$ satisfies (A1$'$) for every $\omega\in\hat\Omega$. 

Let us fix $\omega\in\hat\Omega$.  Take $x\in\{a(\cdot,\omega)=0\}$. Then either $|p^-_{\lambda_0}(x,\omega)-\hat p(x,\omega)|\geqslant \eps_0/2$ or   $|p^+_{\lambda_0}(x,\omega)-\hat p(x,\omega)|\geqslant \eps_0/2$. Let us assume for definiteness that the first alternative holds. Then 
\[
\lambda_0-\hat\lambda(x,\omega)
=
H(x,p^-_{\lambda_0}(x,\omega),\omega)-H(x,\hat p(x,\omega),\omega)
\geqslant 
\eta |p^-_\lambda(x,\omega)-\hat p(x,\omega)|
\geqslant 
\frac{\eta\eps_0}{2}.
\]
This implies that 
\begin{equation}\label{eq2 lower bound}
\lambda_0-\frac{\eta\eps_0}{2}\geqslant \sup_{x\in\{a(\cdot,\omega)=0\}}\hat\lambda(x). 
\end{equation}
Let us fix $\mu\in (\lambda_0-{\eta\eps_0}/{2},\lambda_0)$, $\lambda>\lambda_0$ and  $\eps>0$.  
We want to show that there exists a Lipschitz function $v$ satisfying the following inequality in the viscosity sense 
\begin{equation}\label{claim lower bound}
a(x,\omega)v''+H(x,v',\omega)> \mu-2\eps\qquad\hbox{in $\R$}
\end{equation}
and such that  $v'=(u^+_\lambda)'(\cdot,\omega)$ in $(-\infty, -L)$ and $v'=(u^-_\lambda)'(\cdot,\omega)$ in $(L,+\infty)$ for some $L>0$. This is enough to conclude. In fact,  in view of Proposition \ref{prop corrector argument}-(ii), this would imply that 
\[
\HV^L(H)(\theta)\geqslant  \mu-2\eps\qquad\hbox{for every $\theta \in [\theta^-(\lambda_0),\theta^+(\lambda_0)]
\subsetneq [\theta^-(\lambda),\theta^+(\lambda)]$}.
\]
The assertion follows from this by sending $\eps\to 0^+$ and then $\mu\to \lambda_0^-$. 

In order to define such a subsolution $v$, we first observe that the functions $u_\lambda^{\pm}(\cdot,\omega)$ are in $\CC^1(\R)\cap \CC^2(A)$ and (classical) pointwise solutions of equation \eqref{eq2 cellPDE} in $\R$. Let us set $f_1:=(u^-_\lambda)'(\cdot,\omega)$ and 
$f_2:=(u^+_\lambda)'(\cdot,\omega)$ on $\R$. In view of Proposition \ref{appC prop monotonicity derivatives}, $f_2>f_1$ in $\R$ and each $f_i$ is a classical solution of the  ODE 
\begin{equation}\label{ODE lower bound}
a(x,\omega)f'+H(x,f,\omega)= \lambda\qquad\hbox{in $A$.}
\end{equation}
Pick $R>1+\max\{\|f_1\|_{L^\infty(\R)} ,\|f_2\|_{L^\infty(\R)}\}$ and denote by $C_R$ the Lipschitz constant of $H(\cdot,\cdot,\omega)$ on $\R\times[-R,R]$. We proceed to prove the following claim:\smallskip

\noindent{\bf Claim:} {\em there exists a connected component $J$ of $A$, a pair of points 
$\hat x<\hat y$ in $J$ and a function $g:[\hat x,\hat y]\to\R$ which solves the ODE
\begin{equation}\label{eq ODE lower bound}
a(x,\omega)f'+H(x,f,\omega)=\mu\qquad\hbox{in $(\hat x,\hat y)$ }
\end{equation}
and satisfies 
\begin{equation}\label{eq bounds}
f_1(\cdot)<g<f_2(\cdot)\quad\hbox{in $(\hat x,\hat y)$},
\quad g(\hat x)=f_2(\hat x),
\quad g(\hat y)=f_1(\hat y).
\end{equation}
}

Let us assume the Claim false and let us argue by contradiction. Let $J=(\ell_1,\ell_2)$ be a connected component of the set $A$. 
For  fixed $n\in\N$, let us denote by $  g_n:[\ell_1+1/n,b)\to\R$ the unique solution of the ODE \eqref{eq ODE lower bound} in $[\ell_1+1/n,b)$
satisfying $  g_n(\ell_1+1/n)=f_2(\ell_1+1/n)$, for some $b\in (\ell_1+1/n,\ell_2)$. The existence of such a $b$ is guaranteed by the  classical Cauchy-Lipschitz Theorem. Let us denote by $I$ the maximal interval of the form $(\ell_1+1/n,b)\subset J$ where such a $  g_n$ is defined. 
It is easily seen that $  g_n<f_2(\cdot)$ on $I$: since $\mu<\lambda$ and $a(\cdot,\omega)>0$ on $J$, we have $g_n'(x_0)<f_2'(x_0)$ at any possible point $x_0\in I$ where $g_n(x_0)=f_2(x_0)$, so in fact this equality occurs at $x_0=\ell_1+1/n$ only. Let us denote by $y_n:=\sup\{y\in I\,:\,  g_n>f_1(\cdot)\ \hbox{in $(\ell_1+1/n,y)$}\}$. 
Since we are assuming the Claim false, we must have $y_n=\ell_2$, in particular $f_1(\cdot)<  g_n<f_2(\cdot)$ in $(\ell_1+1/n,\ell_2)$. Being $a(\cdot,\omega)>0$ in $J$, we infer from 
\eqref{eq ODE lower bound} that the functions $  g_n$ are locally equi-Lipschitz and equi-bounded on their domain of definition. Up to extracting a subsequence, we derive that the functions $  g_n$ locally uniformly converge  to a function $  g$ in $J=(\ell_1,\ell_2)$, and also locally in the $C^1$ norm being each $  g_n$ a solution of \eqref{eq ODE lower bound} in $(\ell_1+1/n,\ell_2)$. Hence $  g$ is a solution of \eqref{eq ODE lower bound} in $J$ satisfying 
\begin{equation*}
f_1(\cdot)<   g< f_2(\cdot)\qquad\hbox{in $J$,}
\end{equation*}
where the strict inequality comes from that fact that $  g$ and $f_i$ are solutions of different ODEs. The argument is analogous to the one employed above.  Since this need to happen for every connected component $J$ of $A$, we have found a bounded function $g\in\CC^1(A)$ which solves the ODE \eqref{eq ODE lower bound} in $A$. Let us extend $g$ to the whole $\R$ by setting $g(x)=p^+_\mu(x,\omega)$ for every $x\in \{a(\cdot,\omega)=0\}$, which is possible since $\mu>\sup_{x\in\{a(\cdot,\omega)=0\}}\hat\lambda(x)$ in view of \eqref{eq2 lower bound}. Then the function $w(x):=\int_0^x g(z)\,dz$, $x\in\R$, is a Lipschitz continuous viscosity supersolution of \eqref{eq2 cellPDE} with $\mu$ in place of $\lambda$ by Theorem \ref{appC teo supersolution}. This contradicts the definition of $\lambda_0=\lambda_0(\omega)$ since $\mu<\lambda_0$. The Claim is thus proved.\smallskip 

We derive that there exists a connected component $J=(\ell_1,\ell_2)$ of $A$, a pair of points 
$\hat x<\hat y$ in $J$ and a function $g:[\hat x,\hat y]\to\R$ which solves equation \eqref{eq ODE lower bound} in 
$(\hat x,\hat y)$  and satisfies \eqref{eq bounds}.
%\begin{equation}
%f_1(\cdot)<g<f_2(\cdot)\quad\hbox{in $(\hat x,\hat y)$},
%\quad g(\hat x)=f_2(\hat x),
%\quad g(\hat y)=f_1(\hat y).
%\end{equation}
Let us extend the function $g$ to the whole $\R$ by setting $g=f_2$ on $(-\infty,\hat x)$ and $g=f_1$ on $(\hat y,+\infty)$. 
Take a standard sequence of even convolution kernels $\rho_n$ supported in $(-1/n,1/n)$ and set $g_n:=\rho_n*g$. Let us pick $r>0$ such that the the set $(\hat x-2r,\hat y+2r)$ is compactly contained in $J$
and choose $n\in\N$ big enough so that $1/n<r$ and $\|g-g_n\|_{_{L^\infty(\hat J)}}<1$, where we have denoted by $\hat J$ the interval $(\hat x-r,\hat y+r)$ to ease notation.
We claim that we can choose $n$ big enough such that 
\begin{equation} \label{claim2 lower bound}
a(x,\omega)g_n'+H(x,g_n,\omega)>\mu-\eps\qquad\hbox{in $(\hat x-r,\hat y+r)$.}
\end{equation}
To this aim, first observe that the map $x\mapsto H(x,g(x),\omega)$ is $K$-Lipschitz continuous in 
$(\hat x-2r,\hat y+2r)$, for some constant $K$, due to the fact that $g$ is $C^1$ on $[\hat x-2r,\hat y+2r]\subset J$. 
For every $x\in \hat J=(\hat x-r,\hat y+r)$ and $|y|\leqslant 1/n$ we have 
\begin{eqnarray*}
H(x,(\rho_n*g)(x),\omega) 
&\geqslant&
H(x,g(x),\omega)-C_R\|g-g_n\|_{_{L^\infty(\hat J)}}\\
&\geqslant& 
H(x-y,g(x-y),\omega)-C_R\|g-g_n\|_{_{L^\infty(\hat J)}}-\dfrac{K}{n},
\end{eqnarray*}
hence
\begin{equation*}\label{eq inequality 1}
H(x,g_n(x),\omega)
\geqslant
\int_{-1/n}^{1/n} \rho_n(y)H(x-y,g(x-y),\omega)\, dy -C_R\|g-g_n\|_{_{L^\infty(\hat J)}}-\dfrac{K}{n}.
%\qquad\hbox{for all $x\in (\hat x-r,\hat y+r)$}.
\end{equation*}
For all $x\in (\hat x-r,\hat y+r)$ we have
\begin{eqnarray*}
a(x,\omega)g_n'(x)+H(x,g_n(x),\omega)
&\geqslant& 
-C_R\|g-g_n\|_{_{L^\infty(\hat J)}}-\dfrac{K+2\kappa \esssup\limits_{[\hat x-2r,\hat y+2r]}|g'|}{n}\\
&+& \int_{-1/n}^{1/n} \Big( a(x-y,\omega)g'(x-y)+H(x-y,g(x-y),\omega)  \Big)\,\rho_n(y)\, dy \\
&>& 
\mu -\eps
\end{eqnarray*}
for $n\in\N$ big enough, where for the last inequality we have used the fact that $g$ satisfies \eqref{eq ODE lower bound} in $(\hat x,\hat y)$,  it satisfies \eqref{eq ODE lower bound} in $J\setminus [\hat x,\hat y]$ with $\lambda$ in place of $\mu$, and $\|g-g_n\|_{_{L^\infty(\hat J)}}\to 0$ as $n\to +\infty$. We recall that $2\kappa$ is a Lipschitz constant of the diffusion coefficient $a$. 
Let us now take $\xi\in C^1(\R)$ such that 
\[
0\leqslant \xi \leqslant 1\quad\hbox{in $\R$,}
\quad
\xi\equiv 0\quad\hbox{in $(-\infty,\hat x-r]\cup [\hat y+r,+\infty)$},
\quad
\xi\equiv 1\quad\hbox{in $[\hat x-r/2,\hat y+r/2]$},
\]
and we set $g_\eps(x):=\xi(x) g_n(x)+(1-\xi(x))g(x)$ for all $x\in\R$. The function $g_\eps$ satisfies $g_\eps=f_2$ in $(-\infty,\hat x-r]$, $g_\eps=f_1$ in $[\hat y+r,+\infty)$ and it is of class $C^1$ in $J$. We will show that we can choose $n$ in 
$\N\cap (1/r,+\infty)$ big enough so that  $g_\eps$ satisfies
\begin{equation}\label{claim3 lower bound}
a(x,\omega)g_\eps'+H(x,g_\eps,\omega)> \mu-2\eps\qquad\hbox{in $\hat J=(\hat x-r,\hat y+r)$.}
\end{equation}
For notational simplicity, we momentarily suppress $(x,\omega)$ from some of the notation below and observe that
\begin{align*}\label{eq:concur}
ag_\eps' + H(x,g_\eps,\omega) &= \xi\big(a g_n' + H(x,g_n,\omega)\big) + (1-\xi)\big(a g ' + H(x,g,\omega )\big)
\nonumber\\
&\quad + \xi\big( H(x,\xi g_n + (1-\xi) g,\omega) -H(x,g_n,\omega)\big) \\
&\quad +
(1-\xi)\big( H(x,\xi g_n + (1-\xi) g,\omega) - H(x, g,\omega )\big) +a\xi'( g_n  -  g) \nonumber\\
&\quad 
>\mu-\eps-
\big(2 C_R+\|\xi'\|_\infty\big) \|g-g_n\|_{_{L^\infty(\hat J)}},
\end{align*}
so inequality \eqref{claim3 lower bound} follows for $n\in\N$ big enough since $\|g-g_n\|_{_{L^\infty(\hat J)}}\to 0$ as $n\to +\infty$. 
The sought subsolution $v$ is defined by setting $v(x):=\int_0^x g_\eps(z)\, dz$ for all $x\in\R$. Indeed, in view of the properties enjoyed by $g_\eps$, it is easily seen that 
there exists constants $c^+,\,c^-$ such that 
\begin{equation}\label{eq3 lower bound}
v(x)=u^+_\lambda(x,\omega)+c^+\quad \hbox{in $(-\infty,\hat x-r]$,}
\qquad
v(x)=u^-_\lambda(x,\omega)+c^-\quad  \hbox{in $[\hat y+r,+\infty)$,}
\end{equation}
and $v$ satisfies \eqref{claim lower bound} in view of \eqref{claim2 lower bound}, \eqref{eq3 lower bound}  and the fact that $\lambda>\mu$. 
\end{proof}

\subsection{Proof of Theorem \ref{thm:genhom2}}
For every $\lambda>\lambda_0$ we have, in view of Theorem \ref{teo1 random correctors} and of Proposition \ref{prop consequence existence corrector}, 
\[
\HV(\theta):=\HV^L(H)(\theta)=\HV^U(H)(\theta)=\lambda
\qquad 
\hbox{if\quad $\theta \in \{\theta^-(\landa),\theta^+(\landa)\}$.}
\]
Furthermore, in view of Theorems \ref{teo1 upper qc} and \ref{teo lower bound} and of the fact that $\big[\theta^-(\lambda_0),\theta^+(\lambda_0)\big]\subsetneq \big[\theta^-(\lambda),\theta^+(\lambda)\big]$ for every $\landa>\landa_0$, 
\[
\HV(\theta):=\HV^L(H)(\theta)=\HV^U(H)(\theta)=\lambda_0
\qquad 
\hbox{for all\quad $\theta \in \big[\theta^-(\lambda_0),\theta^+(\lambda_0)\big]$.}
\]
In view of Propositions \ref{prop theta functions}  and \ref{appB prop HF}, this defines a function $\HV(H):\R\to [\lambda_0,+\infty)$ which is superlinear and locally Lipschitz. Homogenization of equation \eqref{eq:introHJ} follows in view of \cite[Lemma 4.1]{DK17}.
Note that $\HV(G)$ is strictly decreasing in $(-\infty,\theta^-(\lambda_0)]$ and strictly increasing  in 
$[\theta^+(\lambda_0), +\infty)$ in view of Proposition \ref{prop theta functions}, and constant on $\big[\theta^-(\lambda_0),\theta^+(\lambda_0)\big]$. In particular, it is strictly quasiconvex, possibly except at the bottom, where it present a flat part whenever $\theta^-(\lambda_0)<\theta^+(\lambda_0)$. \qed

\appendix

\section{Proof of Proposition \ref{prop density}}\label{app:density}

Let $H:\R\times\R\times\Omega\to\R$ be a stationary Hamiltonian satisfying $H(\cdot,\cdot,\omega)\in\Ham$ for every $\omega$, for some fixed constants $\alpha_0,\alpha_1>0$ and $\gamma>1$. We start by remarking what follows. 

\begin{lemma}\label{appA lemma Lipschitz min H}
The function $V(x,\omega):=\min_{p\in\R} H(x,p,\omega)$ is a bounded and Lipschitz stationary function. More precisely:
\begin{itemize}
\item[\em (i)] \quad $-1/\alpha_0\leqslant V(x,\omega)\leqslant \alpha_1$\quad for every $(x,\omega)\in\R\times\Omega$;\smallskip
\item[\em (ii)] there exists a constant $\hat\kappa=\hat\kappa(\alpha_0,\alpha_1,\gamma)$, only depending on $\alpha_0,\alpha_1>0$ and $\gamma>1$, such that 
\[
V(\cdot,\omega)\ \ \hbox{is $\hat\kappa$-Lipschitz continuous in $\R$,\  \ for every $\omega\in\Omega$.} 
\]
\end{itemize}
\end{lemma}

\begin{proof}
It is easily seen that $V$ is stationary. 
Let us denote by $\hat\M(x,\omega)$ the set of minimizers of the function $p\mapsto H(x,p,\omega)$ in $\R$, for every fixed $(x,\omega)$.  From (H1) we get 
\[
\alpha_0|\hat p|^\gamma-\frac{1}{\alpha_0}
=
H(x,\hat p,\omega)
\leqslant 
H(x,0,\omega)
\leqslant 
\alpha_1
\qquad
\hbox{for every $\hat p\in\hat\M(x,\omega)$}.
\]
This implies $-1/{\alpha_0}\leqslant V(x,\omega)\leqslant \alpha_1$ and $\hat \M(x,\omega)\subseteq [-\hat R,\hat R]$ for every $(x,\omega)\in\R\times\Omega$, where 
\[
\hat R=\hat R(\alpha_0,\alpha_1,\gamma):=\left(\frac{1+\alpha_1\alpha_0}{\alpha_0^2}\right)^{\frac{1}{\gamma}}.
\]
Let us denote by $C=C(\alpha_1,\gamma,\hat R)$ a Lipschitz constant of $H(\cdot,\cdot,\omega)$ in $\R\times [-\hat R,\hat R]$. For every fixed $\omega\in\Omega$ and every $x,y\in\R$ we have 
\[
V(x,\omega)-V(y,\omega) \leqslant H(x,\hat p_y)-H(y,\hat p_y) \leqslant C|x-y|
\qquad
\hbox{with\  $\hat p_y\in\hat\M(y,\omega)$.}
\]
The proof is complete.
\end{proof}

Let us set \ 
$
G(x,p,\omega):=H(x,p,\omega)-V(x,\omega)\quad\hbox{with\quad $V(x,\omega):=\min_{p\in\R} H(x,p,\omega)$} 
$
for every $(x,p,\omega)\in\R\times\R\times\Omega$. Then $H(x,p,\omega)=G(x,p,\omega)+V(x,\omega)$, where $G(\cdot,\cdot,\omega)\in\Hamall(\tilde\alpha_0,\tilde\alpha_1,\gamma)$ for every $\omega\in\Omega$, for possible different constants 
$\tilde\alpha_0\geqslant \alpha_0$ and $\tilde\alpha_1\geqslant \alpha_1$. Note that 
\begin{equation}\label{appA eq min G}
\min_{p\in\R} G(x,p,\omega)=0\qquad\hbox{for every $(x,\omega)\in\R\times\Omega$.}
\end{equation}
Fix $n\in\N$ and for every $(x,p,\omega)\in\R\times\R\times\Omega$ set 
\[
\tilde G_n(x,p,\omega):=\max\{2/n,G(x,p,\omega)\}\quad\hbox{and}\quad \tilde H_n(x,p,\omega):=\tilde G_n(x,p,\omega)+V(x,\omega).
\]
Up to choosing larger constants $\tilde\alpha_0,\,\tilde\alpha_1$, if necessary, we have that $\tilde G_n(\cdot,\cdot,\omega)\in\Hamall (\tilde\alpha_0,\tilde\alpha_1,\gamma)$ for every $\omega\in\Omega$ and $n\in\N$. 
It is easily seen that 
\[
|H(x,p,\omega)-\tilde H_n(x,p,\omega)|=|\tilde G_n(x,p,\omega)-G(x,p,\omega)|\leqslant\frac2n\qquad\hbox{for every $(x,p,\omega)\in\R\times\R\times\Omega$.}
\]
Hence it suffices to prove Proposition \ref{prop density} with $\tilde G_n$ in place of $H$. 

\begin{lemma}\label{appA lemma measurable selection}
Let $n\in\N$ and set \ $Z_{2n}(\omega):=\{p\in\R\,:\, G(0,p,\omega)\leqslant 1/n\,\}$ \ for every $\omega\in\Omega$. Then there exists a measurable selection for the set-valued map $Z_{2n}$, i.e., a measurable function $\tilde p_n:\Omega\to\R$ such that $\tilde p_n(\omega)\in Z_{2n}(\omega)$ for every $\omega\in\Omega$.  
\end{lemma}

\begin{proof}
For each $\omega$, the set $Z_{2n}(\omega)$ is a compact and convex subset of $\R$. Furthermore, 
by \cite[Proposition 3.2]{DS09}, we know that $\{\omega\,:\,Z_{2n}(\omega)\cap K\not=\emptyset\}\in\F$ for every compact set $K$ in $\R$. The assertion follows by applying the  Kuratowski–Ryll-Nardzewski measurable selection theorem \cite{KuRy65}.
\end{proof}

We define a jointly measurable function $p_n:\R\times\Omega\to\R$ by setting $p_n(x,\omega):=\tilde p_n(\tau_x\omega)$ for every $(x,\omega)\in\R\times\Omega$. Since $p_n$ is neither Lipschitz nor continuous with respect to $x$, we need to regularize it as explained in the proof of the next lemma.

\begin{lemma}\label{appA lemma regularization}
For every fixed $n\in\N$, there exists a jointly measurable and stationary function $\hat p_n:\R\times\Omega\to\R$ and a constant $\hat\kappa_n$ such that 
\begin{itemize}
\item[\em (i)] \quad $\hat p_n(x,\omega)\in \{p\in\R\,:\,G(x,p,\omega)\leqslant 2/n\,\}$\quad for every $(x,\omega)\in\R\times\Omega$;\smallskip
\item[\em (ii)] \quad $\hat p_n(\cdot,\omega)$\ \  is $\hat\kappa_n$-Lipschitz continuous in $\R$,\  \ for every $\omega\in\Omega$. 
\end{itemize}
\end{lemma}

\begin{proof}
For every $k\in\N$, we set $Z_k(x,\omega):=\{p\in\R\,:\,G(x,p,\omega)\leqslant 2/k\}$ for every $(x,\omega)\in\R\times\Omega$.  
We know by construction that $p_n(x,\omega)\in Z_{2n}(x,\omega)$ for every $(x,\omega)\in\R\times\Omega$. From the fact that $G(\cdot,\cdot,\omega)\in\Hamall(\tilde\alpha_0,\tilde\alpha_1,\gamma)$ for every $\omega\in\Omega$, we infer that there exists a radius $\tilde R=\tilde R(\tilde\alpha_0,\gamma)>0$ such that $Z_1(x,\omega)$ is contained in $(-\tilde R,\tilde R)$. Let us denote by $C=C(\tilde R)$ a Lipschitz constant of $G(\cdot,\cdot,\omega)$ in $\R\times [-\tilde R,\tilde R]$, for every $\omega\in\Omega$. Let us set $r_n:=1/(2n C)$. We claim that
\begin{equation*}\label{appA eq regularization}
Z_{2n}(y,\omega)\subset Z_n{(x,\omega)}\qquad\hbox{for all $x,y\in\R$ with $|x-y|<r_n$. }
\end{equation*}
%Let us first show that the inclusion \eqref{appA eq regularization} holds with $\{0\}$ in place of the interval $[-r_n,r_n]$. 
Indeed, if $p\in Z_{2n}(y,\omega)$, then
\[
G(x,p,\omega)<G(y,p,\omega)+Cr_n< \frac1n+\frac{1}{2n}=\frac{3}{2n}<\frac2n
\qquad\hbox{for every $|y-x|<r_n$.}
\]
%Let us prove claim \eqref{appA eq regularization}. Let $p_1\in\partial Z_{2n}(y,\omega)$ and $p_2\in\partial Z_{n}(x,\omega)$. Then $G(y,p_1,\omega)=1/n$ and $G(x,p_2,\omega)=2/n$. We infer 
%\[
%\frac1n=\frac{2}{n}-\frac1n=G(x,p_2,\omega)-G(y,p_1,\omega)\leqslant C(|p_2-p_1|+|x-y|),
%\] 
%hence if $|y-x|<r_n$ we get 
%\[
%|p_2-p_1|\geqslant \frac{1}{nC}-|x-y|=2r_n-|x-y|>r_n.
%\]
%This proves \eqref{appA eq regularization}. 
Take now a standard positive and even convolution kernel $\rho$ supported in $(-r_n,r_n)$ and set 
\[
\hat p_n(x,\omega):=(\rho * p_n)(x,\omega)=\int_{-r_n}^{r_n} \rho(y)p_n(x-y,\omega)\,dy\qquad\hbox{for every $(x,\omega)\in\R\times\Omega$}.
\] 
Since $p_n(x-y, \omega)\in Z_{2n}(x-y,\omega)\subset Z_n(x,\omega)$ for every $(x,\omega)\in\R\times\Omega$ and $|y|<r_n$, from the convexity of $Z_n(x,\omega)$ we get that $\hat p_n(x,\omega)\in Z_n(x,\omega)$ for all $(x,\omega)\in\R\times\Omega$. The function $\hat p_n$ is jointly measurable and stationary by construction. Let us prove assertion (ii). From the fact that 
$|p_n(z,\omega)|\leqslant \tilde R$ for every $(z,\omega)\in \R\times\Omega$, we infer 
\[
|\hat p_n'(x,\omega)|=|(\rho' *p_n)(x,\omega)|\leqslant \int_{-r_n}^{r_n} |\rho'(y)p_n(x-y,\omega)|\, dy
\leqslant
\tilde R\|\rho'\|_{L^1(\R)}\quad\hbox{for every $(x,\omega)\in\R\times\Omega$.}
\]   
Assertion (ii) follows by setting $\tilde \kappa_n:=\tilde R\|\rho'\|_{L^1(\R)}$. Note that the choice of $\rho$ depends on $n$ since we are requiring that the support of $\rho$ is contained in $(-r_n,r_n)$. 
\end{proof}

For every $n\in\N$, we define a stationary Hamiltonian $G_n:\R\times\R\times\Omega\to\R$ by setting 
\[
G_n(x,p,\omega):=G(x,p,\omega)+\eta_n |p-\hat p_n(x,\omega)|^4\qquad\hbox{for all $(x,p,\omega)\in\R\times\R\times\Omega$,}
\] 
with $\eta_n$ suitably chosen in the interval $(0,1/n)$. It is clear that each $G_n$ satisfies (sqC) with $\eta:=\eta_n$. 
Furthermore, 
\[
\lim_n\|G(\cdot,\cdot,\omega)-G_n(\cdot,\cdot,\omega)\|_{L^\infty\left(\R\times [-R,R]\right)}=0\qquad
\hbox{for every $R>0$ and $\omega\in\Omega$.}
\]
A tedious but otherwise standard computation shows that the constants $\eta_n$ can be chosen in such a way that $G_n(\cdot,\cdot,\omega)\in\Hamall(\ol\alpha_0,\ol\alpha_1,\ol \gamma,\eta_n)$ for every $\omega$,  where $\ol\alpha_i=\tilde\alpha_i+1$ for each $i\in\{1,2\}$ and $\ol\gamma:=\max\{\gamma,4\}$.   

\section{PDE results}\label{app:PDE}

In this appendix, we collect some known PDE results that we need in the paper. Throughout this section, we will denote by 
$\D{UC}(X)$, $\D{LSC}(X)$ and $\D{USC}(X)$ the space of uniformly continuous, lower semicontinuous and upper semicontinuous real functions on a metric space $X$, respectively. 
We will denote by $H$ a continuous function defined on $\R\times\R$. If not otherwise stated, we shall assume that 
$H$ belongs to the class $\Ham$  introduced in Definition \ref{def:Ham}, for some constants $\alpha_0,\alpha_1>0$ and $\gamma>1$. 

We will assume that $a:\R\to [0,1]$ is a function satisfying the following assumption, for some constant $\kappa  > 0$: 
\begin{itemize}
\item[(A)]  $\sqrt{a}:\R\to [0,1]$\ is $\kappa $--Lipschitz continuous.
\end{itemize}
Note that (A) implies that $a$ is $2\kappa$--Lipschitz in $\R$.\smallskip 

We record here  the following trivial remark, which is used in some of our arguments.  
\begin{lemma}\label{lemma 1/a}
Let $I=(\ell_1,\ell_2)$ be a bounded interval such that $a>0$ in $I$ and $a(\ell_1)=a(\ell_2)=0$. For every $x\in I$ we have
\[
\lim_{y\to\ell_1^+} \int_y^x \frac{1}{a(z)}\, dz=+\infty,
\qquad
\lim_{y\to\ell_2^-} \int_x^y \frac{1}{a(z)}\, dz=+\infty.
\]
\end{lemma}

\begin{proof}
For every $z\in I$ we have
\[
\frac{1}{a(z)}
=
\frac{1}{a(z)-a(\ell_2)}
\geqslant 
\frac{1}{2\kappa |z-\ell_2|},
\qquad
\frac{1}{a(z)}
=
\frac{1}{a(z)-a(\ell_1)}
\geqslant 
\frac{1}{2\kappa |z-\ell_1|},
\]
By integrating the above inequalities we get the assertion.
\end{proof}

%
% satisfying the 
%following set of assumptions, for some fixed constants $\alpha_0,\alpha_1>0$ and $\gamma>1$:\smallskip
%%
%\begin{itemize}
%\item[(H1)] $\alpha_0|p|^\gamma-1/\alpha_0\leqslant H(x,p)\leqslant\alpha_1(|p|^\gamma+1)$\quad for all ${x,p\in\R}$;\medskip
%%
%\item[(H2)] $|H(x,p)-H(x,q)|\leqslant\alpha_1\left(|p|+|q|+1\right)^{\gamma-1}|p-q|$\quad for all $p,q,x\in\R$;\medskip
%%
%\item[(H3)] $|H(x,p)-H(y,p)|\leqslant\alpha_1\left(|p|^\gamma+1\right)|x-y|$\quad for all $x,y,p\in\R$. \medskip
%\end{itemize}
%%
%

\subsection{Stationary equations} Let us consider a stationary viscous HJ equation of the form
\begin{equation}\label{eq PDE}
	a(x)u''(x)+H(x,u')=\lambda\qquad \hbox{ in $\R$,}
		\tag{HJ$_\lambda$}
\end{equation}
where $\lambda\in\R$, the nonlinearity $H$ belongs to $\Ham$, and  $a:\R\to [0,1]$ satisfies condition (A). 
The following holds.

\begin{prop}\label{prop regularity solutions}
Let $u\in\CC(\R)$ be a viscosity solution of \eqref{eq PDE}. Let us assume that $H\in\Ham$ for some constants $\alpha_0,\alpha_1>0$ and $\gamma>1$. Then 
\[
|u(x)-u(y)| \leqslant K |x-y|
\qquad\hbox{for all $x,y\in\R$,}
\]
where $K>0$ is given explicitly by 
\begin{equation}\label{eq Lipschitz bound}
K:=C\left( 
		\left(
			\kappa  
			\dfrac{\sqrt{1+\alpha_1+|\lambda|}}{\alpha_0}
		\right)^{\frac{2}{\gamma-1}} + 
	\left(
	\dfrac{1+\lambda\alpha_0}{\alpha^2_0}
	\right)^{\frac1\gamma}
\right)
\end{equation}
with $C>0$ depending only on $\gamma$.
Furthermore, $u$ is of class $C^2$ (and hence a pointwise solution of \eqref{eq PDE}) in every open interval $I$ where $a(\cdot)$ is strictly positive. 
\end{prop}

\begin{proof}
The Lipschitz character of $u$ is direct consequence of \cite[Theorem 3.1]{AT}, to which we refer for a proof.  
Let us now assume that $a(\cdot)$ is strictly positive on some open interval $I$. Without loss of generality, we can assume that $I$ is bounded and $\inf_I a>0$. From the Lipschitz character of $u$ we infer that $-C\leq u''\leq C$ in $I$ in the viscosity sense 
for some constant $C>0$, or, equivalently, in the distributional sense, in view of  \cite{Is95}. Hence, 
$u'' \in L^\infty(I)$. The elliptic regularity theory, see \cite[Corollary 9.18]{GilTru01}, ensures that $u\in W^{2,p}(I)$ for any $p>1$
and, hence, $u\in \D{C}^{1,\sigma}(I)$ for any $0<\sigma<1$. Since $u$ is a viscosity solution to \eqref{eq PDE} in $I$, 
by Schauder theory \cite[Theorem 5.20]{HL97}, we conclude that $u\in \D{C}^{2,\sigma}(I)$ for any $0<\sigma<1$.
\end{proof}

The next result, together with Proposition \ref{prop regularity solutions}, yields in particular that any continuous viscosity solution of \eqref{eq PDE} satisfies the viscous HJ 
equation pointwise on the set of its differentiability points.

\begin{prop}\label{prop pointwise solution}
Let $u\in\CC(\R)$ be viscosity supersolution (respectively, subsolution) of \eqref{eq PDE}. If $u$ is differentiable at a point $x_0\in\R$ where  $a(x_0)=0$, then 
\[
H(x_0,u'(x_0)) \leqslant \lambda \qquad \hbox{(resp., $H(x_0,u'(x_0))\geqslant \lambda$).}
\]
\end{prop}

\begin{proof}
Without any loss in generality, we can assume $x_0=0$ and $u(0)=u'(0)=0$. Then 
\begin{equation}\label{eq small o}
|u(x)|\leqslant |x|\omega(|x|)\qquad\hbox{in a neighborhood of $x=0$,}
\end{equation} 
where $\omega$ is a continuity modulus. Let us assume that $u$ is a supersolution. Fix $\eps>0$ and $r>0$ and  set $\varphi_r(x):=-\eps {x^2}/{r}$. By \eqref{eq small o}, there exists $r_0>0$ such that $\omega(r_0)<\eps$. For every $r\leqslant r_0$ we have 
\[
\max_{\partial B_r} (\varphi_r-u) 
\leqslant
r(-\eps +\omega(r))
<0=(\varphi_r-u)(0),
\]
hence there exists $x_r\in B_r$ such that $\varphi_r-u$ attains a local maximum at $x_r$. Being $u$ a supersolution of \eqref{eq PDE}, we infer
\[
-a(x_r)\dfrac{2\eps}{r}+H\left(x_r,-2\eps\frac{x_r}{r}\right)\leqslant \lambda.
\]
Now $|x_r/r|\leqslant 1$ and 
$\displaystyle \left|\frac{a(x_r)}{r}\right| =  \left|\frac{a(x_r)-a(0)}{r}\right|\leqslant 2\kappa  \left|\frac{x_r}{r}\right|\leqslant 2\kappa$. By sending $r\to 0^+$ and by extracting convergent subsequences, we can find $|p_\eps|\leqslant 1$ and $|\alpha_\eps|\leqslant 2\kappa$ such that
\[
2\eps \alpha_\eps +H(0,2\eps p_\eps)\leqslant \lambda.
\]
By sending $\eps\to 0^+$ we finally obtain $H(0,0)\leqslant \lambda$, as it was to be shown. The case when $u$ is a subsolution can be handled analogously.
\end{proof}

We shall also need the following H\"older estimate for supersolutions of \eqref{eq PDE}.

\begin{prop}\label{prop Holder estimate}
Let us assume that $H\in\Ham$ with $\gamma>2$. Let $u\in\CC(\R)$ be a supersolution of \eqref{eq PDE} for some 
$\lambda\in\R$.  Then 
\begin{equation*}%\label{eq Holder bounds}
|u(x)-u(y)| \leqslant K |x-y|^{\frac{\gamma-2}{\gamma-1}}\qquad\hbox{for all $x,y\in\R$,}
\end{equation*}
where $K>0$ is given explicitly by 
\begin{equation}\label{eq Holder bound}
K:=C\left( 
		\left(
			\dfrac{1}{\alpha_0}
		\right)^{\frac{1}{\gamma-1}} + 
	\left(
	\dfrac{1+\lambda\alpha_0}{\alpha^2_0}
	\right)^{\frac1\gamma}
\right)
\end{equation}
with $C>0$ depending only on $\gamma$.
\end{prop}

\begin{proof}
The function $v(x):=-u(x)$ is a viscosity subsolution of \eqref{eq PDE} with $-a(\cdot)$ in place of $a(\cdot)$ and $\check H(x,p):=H(x,-p)$ in place of $H$. By the fact that $a\leqslant 1$ and $H\in\Ham$, we derive that $v$ satisfies the following inequality in the viscosity sense:
\[
-|v''|+\alpha_0|v'|^\gamma\leqslant \lambda+\dfrac{1}{\alpha_0}\quad\hbox{in $\R$}.
\]
The conclusion follows by applying \cite[Lemma 3.2]{AT}.
\end{proof}

\subsection{Parabolic equations}
Let us consider a parabolic PDE of the form
\begin{equation}\label{appB eq parabolic HJ}
\partial_{t }u=a(x) \partial^2_{xx} u +H(x,\partial_x u), \quad(t,x)\in (0,+\infty)\times\R.
\end{equation}

We start with a comparison principle stated in a form which is the one we need in the paper.

\begin{prop}\label{appB prop comparison}
Suppose $a$ satisfies (A) and  $H\in\D{UC}\left(B_r\times\R\right)$ for every $r>0$. Let $v\in\D{USC}([0,T]\times\R)$ and $w\in\D{LSC}([0,T]\times\R)$ be, respectively, a sub and a supersolution of \eqref{appB eq parabolic HJ} in $(0,T)\times \R$ such that 
\begin{equation}\label{hyp 2}
\limsup_{|x|\to +\infty}\ \sup_{t\in [0,T]}\frac{v(t,x)-\theta x}{1+|x|}\leqslant 0 
\leqslant 
\liminf_{|x|\to +\infty}\ \sup_{t\in [0,T]}\frac{w(t,x)-\theta x}{1+|x|}
\end{equation}
for some $\theta\in\R$. Let us furthermore assume that either $\partial_x v$ or $\partial_x w$ belongs to $L^\infty\left((0,T)\times \R\right)$. Then, 
\[
v(t,x)-w(t,x)\leqslant \sup_{\R}\big(v(0,\,\cdot\,) - w(0,\,\cdot\,)\big)\quad\hbox{for every  $(t,x)\in (0,T)\times \R$.} 
\]
\end{prop}

\begin{proof}
The functions $\tilde v(t,x):=v(t,x)-\theta x$ and $\tilde w(t,x):=w(t,x)-\theta x$ are, respectively, a subsolution and a supersolution of \eqref{appB eq parabolic HJ} in $(0,T)\times \R$ with $H(\,\cdot\,,\theta +\,\cdot\,)$ in place of $H$. The assertion follows by applying \cite[Proposition 1.4]{D19} to $\tilde v$ and $\tilde w$.
\end{proof}

Let us now assume that $H$ belongs to the class $\Ham$ for some fixed constants $\alpha_0,\alpha_1>0$ and $\gamma>1$.
%\smallskip
%%
%\begin{itemize}
%\item[(H1)] $\alpha_0|p|^\gamma-1/\alpha_0\leqslant H(x,p)\leqslant\alpha_1(|p|^\gamma+1)$\quad for all ${x,p\in\R}$;\medskip
%%
%\item[(H2)] $|H(x,p)-H(x,q)|\leqslant\alpha_1\left(|p|+|q|+1\right)^{\gamma-1}|p-q|$\quad for all $p,q,x\in\R$;\medskip
%%
%\item[(H3)] $|H(x,p)-H(y,p)|\leqslant\alpha_1\left(|p|^\gamma+1\right)|x-y|$\quad for all $x,y,p\in\R$. \medskip
%\end{itemize}
%
%
The following holds. 

\begin{theorem}\label{appB teo well posed}
Suppose $a$ satisfies (A) and $H\in\Ham$. Then, for every $g\in\D{UC}(\R)$, there exists a unique function $u\in \D{UC}(\ccyl)$ that solves the equation \eqref{appB eq parabolic HJ}
subject to the initial condition $u(0,\,\cdot\,)=g$ on $\R$. If $g\in W^{2,\infty}(\R)$, then $u$ is Lipschitz continuous in $\ccyl$ and satisfies
\begin{equation*}
\|\partial_t u\|_{L^\infty(\ccyl)}\leqslant K 
\quad\text{and}\quad
\|\partial_x u\|_{L^\infty(\ccyl)}\leqslant K 
\end{equation*}
for some constant $K $ that depends only on $\|g'\|_{L^\infty(\R)}$, $\|g''\|_{L^\infty(\R)}$, $\kappa , \alpha_0,\alpha_1$ and $\gamma$. Furthermore, the dependence of $K $ on $\|g'\|_{L^\infty(\R)}$ and $\|g''\|_{L^\infty(\R)}$ is continuous. 
\end{theorem}

\begin{proof}
A proof of this result when the initial datum $g$ is furthermore assumed to be bounded is given in \cite[Theorem 3.2]{D19}, see also \cite[Proposition 3.5]{AT}. 
This is enough, since we can always reduce to this case by possibly picking a function $\tilde g\in W^{3,\infty}(\R)\cap\CC^\infty(\R)$ such that $\|g-\tilde g\|_{L^\infty(\R)}<1$ (for instance, by mollification) and by considering equation \eqref{appB eq parabolic HJ} with
$\tilde H(x,p):=a(x)(\tilde g)'' +H(x,p + (\tilde g)')$ in place of $H$, and initial datum $g-\tilde g$. 
\end{proof}

For every fixed $\theta\in\R$, we will denote by $u_\theta$ the unique solution of \eqref{appB eq parabolic HJ} in $\D{UC}(\ccyl)$ satisfying  $u_\theta(0,x)=\theta x$ for all $x\in\R$. 
Theorem \ref{appB teo well posed} tells us that $u_\theta$ is Lipschitz in $\ccyl$ and that its Lipschitz constant depends continuously on $\theta$. 
%
%In particular,  for every $r>0$, there exists a constant $\kappa_r>0$, depending only on 
%$r, \hat\kappa_a, \alpha_0,\alpha_1$ and $\gamma$, such that 
%\begin{equation}\label{appB eq Lipschitz estimate}
%\|\partial_t u_\theta\|_{L^\infty(\ccyl)}\leqslant \kappa_r
%\quad\text{and}\quad
%\|\partial_x u_\theta\|_{L^\infty(\ccyl)}\leqslant \kappa_r
%\quad 
%\hbox{for every $\theta\in B_r$.} 
%\end{equation}
%
Let us define 
\begin{eqnarray}\label{appB eq infsup}
	\HF^L(H) (\theta):=\liminf_{t\to +\infty}\ \frac{u_\theta(t,0)}{t}\quad\text{and}
	\quad
	\HF^U(H) (\theta):=\limsup_{t\to +\infty}\ \frac{u_\theta(t,0)}{t}.
\end{eqnarray}
By definition, we have \ $\HF^L(H) (\theta)\leqslant \HF^U(H) (\theta)$\ for all $\theta\in\R$. 
Furthermore, the following holds, see \cite[Proposition B.3]{DKY23} for a proof.

\begin{prop}\label{appB prop HF}
Suppose $a$ satisfies (A) and $H\in\Ham$. Then the functions $\HF^L(H)$ and $\HF^U(H)$ satisfy (H1) and are locally Lipschitz on $\R$. 
\end{prop}

According to \cite[Theorem 3.1]{DK17}, equation \eqref{appB eq parabolic HJ} homogenizes if and only if the functions 
$u^\eps_\theta(t,x):=\eps u_\theta(t/\eps,x/\eps)$ converge, locally uniformly in $\ccyl$, to a continuous function $\overline u_\theta(t,x)$. In this instance, we have $\HF^L(F)(\theta)=\HF^U(F)(\theta)$ and 
\[
\overline u_\theta(t,x)=\theta x+t\HF(F)(\theta),
\quad
\hbox{$(t,x)\in\ccyl$,}
\]
with $\HF(F)(\theta):=\HF^L(F)(\theta)=\HF^U(F)(\theta)$. 

We end this appendix by stating the following stability result for homogenization, see \cite[Theorem B.4]{DKY23} for a proof. 

\begin{theorem}\label{appB teo stability}
Suppose $a$ satisfies (A), the Hamiltonians $H$, $(H_n)_{n\in\N}$ belong to $\Ham$, and
\[
\lim_{n\to +\infty}\|H_n-H\|_{L^\infty(\R\times [-R,R])}=0
\quad
\hbox{for every $R>0$.}
\]
If equation \eqref{appB eq parabolic HJ} homogenizes for each $H_n$ with effective Hamiltonian $\HV(H_n)$, then it homogenizes for $H$ too, with effective Hamiltonian 
\[
\HV(H)(\theta)=\lim_{n\to +\infty} \HV(H_n)(\theta)
\quad
\hbox{for all $\theta\in\R$.}
\]
Furthermore, this convergence is locally uniform in $\R$. 
%=\Fam(\alpha_0,\alpha_1,\gamma)$ for fixed consants $\alpha_0,\alpha_1>0$ and $\gamma>1$.
\end{theorem}

\bibliography{viscousHJ}
\bibliographystyle{siam}
\end{document}